\documentclass[preprint,12pt]{elsarticle}
\usepackage{amsmath,amssymb,graphicx,bbm,bm,delarray,pict2e,ntheorem,color,tikz}
\usepackage{braket}\usepackage{stmaryrd}
\usepackage{hyperref}
\usepackage[margin=1in]{geometry} 

\newtheorem{theorem}{Theorem}[section]
\newtheorem{prop}[theorem]{Proposition}
\newtheorem{lemma}[theorem]{Lemma}

\newtheorem{corollary}[theorem]{Corollary}

\theorembodyfont{\rmfamily}
\newtheorem{remark}[theorem]{Remark}
\newtheorem{example}[theorem]{Example}

\newtheorem{exc}[theorem]{Exercise}
\newtheorem{defn}[theorem]{Definition}

\newenvironment{proof}{\medskip\noindent{\bf Proof. }}{\hfill$\square$\medskip}
\newenvironment{proof*}[1]{\medskip\noindent{\bf Proof of #1.}}{\hfill$\square$\medskip}

\newcommand{\mm}{\mathfrak m}
\newcommand{\Per}{\mathrm{Per}}

\newcommand{\di}{\mathsf{d}}
\def\essinf{\mathop{\mathrm{essinf}}}

\def\A{\mathcal{A}}\def\B{\mathcal{B}}

\def\Rbb{\mathbb{R}}

\def\R{\Rbb}

\def\vol{\text{\rm vol}}

\def\Bd{{\sf Bd}}

\long\def\ignore#1{}

\def\url{}
\journal{}

\begin{document}
\begin{frontmatter}
\title{Choquet-type extensions of set-pair functions, and applications in graph limits, hypergraphs, and metric measure spaces 
}
\author[1]{Dong Zhang
}
\ead{dongzhang@math.pku.edu.cn}
\affiliation[1]{organization={School of Mathematical Sciences},
            addressline={Peking University},
            city={Beijing},
            postcode={100871},
            country={China}}

\begin{abstract}
We propose Choquet-type extensions for set-pair functions, an $L^p$-integration of a family of such extensions, and the associated min-max extension constants, and apply these to investigate variational problems and establish sharp bounds for many combinatorial and geometric quantities. 
Specifically, we use Choquet-type extensions to construct equivalent functional representations of fractional optimization problems with set-variables, which find applications in various settings, such as maxcut and bipartiteness ratio on graph limits, and Cheeger constants on Riemannian manifolds. 
We further establish the $L^p$ integration of a family of Choquet-type extensions, and apply it to derive spectral bounds for conductance and other combinatorial quantities on measure spaces.  
A monotonicity inequality on min-max extension constants is proposed, which unifies classical estimates and uncovers new bounds for many geometric and combinatorial quantities, such as torsional rigidity,  Cheeger constants, Dirichlet $p$-isoperimetric constant, and $p$-Laplacian eigenvalues, in seemingly unrelated underlying structures----including  hypergraphs, graph limits, Riemannian manifolds and metric measure spaces. 
\end{abstract}
\end{frontmatter}
\allowdisplaybreaks[4]


\normalsize
\tableofcontents

\section{Introduction}
The concept of Choquet integral was introduced by Gustave Choquet in his landmark 1953 paper \cite{Choquet} that laid rigorous foundations for the field of potential theory
, making extensive use of set theoretic, measure theoretic and topological methods. 
This was independently proposed by L\'aszl\'o Lov\'asz \cite{Lovasz83} in the context of combinatorial optimization, leading to a profound and useful framework on submodular analysis, and continues to attract widespread attention to this day; for example, the Choquet integral has become a useful tool in game theory and decision theory \cite{Schmeidler}, and has also been extended to general spaces \cite{Cerreia15,Lovasz25}. Recently, necessary and sufficient conditions for submodular functions to be representable as Choquet integrals were provided \cite{Chateauneuf}, and  Poincaré–Sobolev inequalities involving Choquet integrals were established \cite{Harjulehto}, giving rise to a wealth of research findings and applications. 
Very recently, it has been applied to spectral graph theory \cite{Jost/Mulas/Zhang26} and discrete Morse theory \cite{Jost/Zhang24a}, which plays the role of a bridge between discrete problems and continuous methods. 

In this paper, we generalize the original Choquet integral (or Choquet extension, or Lov\'asz extension) to set-pair versions. A motivation is that  many commonly used quantities are defined via set-pair functions rather than setfunctions, e.g., the bipartiteness ratio of a graph \cite{Trevisan2012} and the frustration index of a signed graph
. Such quantities cannot, in general, be represented using the classical Choquet extension of setfunctions. This motivates us to introduce a disjoint-pair Choquet extension for set-pair functions, which 
provides a unified framework for studying combinatorial quantities on signed graphs and their limit objects. 

Another motivation comes from the variational point of view. 
In fact, a Rayleigh-type functional arising from the extension of setfunctions is generally unbounded from below (unless the domain is restricted to a convex 
cone, or the setfunctions satisfy some symmetry conditions), which makes a direct variational method inconvenient. 
In contrast, the extension of set-pair functions yields a globally defined nonnegative Rayleigh-type functional, thereby facilitating the subsequent application of variational methods. 
This is one of the motivations for introducing the disjoint-pair Choquet extension. 
Building on this extension, 
we further propose 
$L^p$ integration of a family of Choquet-type extensions, and the associated min-max extension constants, and apply these to investigate variational problems and establish sharp estimates for many combinatorial and geometric quantities. 
This research line is also applicable to the study of original Choquet extension. 
Thanks to the extended framework for bisubmodular (and more general set-pair) functions, we can incorporate certain quantities involving hypergraphs, graph limits,  Riemannian manifolds, and metric measure spaces into a single framework, thereby enabling a unified study. The main results 
are as follows: 
\begin{enumerate}
\item We propose a machinery  based on Choquet-type extensions, and this gives a way to equivalently transform set-based optimization to functional-based optimization in closed form, which facilitates the subsequent application of the variational method. 
The functional point of view enables many tools to be implemented
. Such equivalent transformations make many problems transparent, also help us understand and compute a wide range of concepts and quantities, such as maxcut, bipartiteness ratio, frustration index, etc; see Theorem \ref{thm:tilde-fg-equal} and its applications in Section \ref{sec:frac}. 
\item 
We establish the $L^p$ integration of a family of Choquet-type extensions and, based on this, we propose a monotonicity inequality that applies not only to many quantities, such as conductance and torsional rigidity, but also to many distinct types of underlying structures, including graph limits, hypergraphs, Riemannian manifolds and metric measure spaces.  
This monotonicity naturally gives rise to many useful estimates; we may therefore regard it as a source of many important inequalities. See Theorems \ref{thm:critical-inequality} and \ref{thm:critical-inequality2} and their applications in Sections \ref{sec:Lp-integration} and \ref{sec:Applications}. 

\item A significant difference from previous work in literature \cite{Choquet,Lovasz83,Lovasz25,Jost/Zhang24a} is that we propose $L^p$ integration of a family of Choquet integrals and the corresponding min-max extension constants (Definitions \ref{def:Phi-integration}, \ref{def:Psi-integration} and \ref{def:global-extension-cons}). This unified framework allows us to apply the extension methods to estimate many combinatorial and geometric quantities; for example, a monotonicity inequality for the min-max eigenvalues of $p$-Laplacians on metric measure spaces is provided in Theorem \ref{thm:mono-mms}; a lower bound for $p$-torsional rigidity, expressed in terms of distance integrals, is derived in Proposition \ref{prop:torsional}. 
In addition, as an analytic framework for analyzing large graphs, the graph limits (e.g. the graphon) have become a meeting point of many areas of mathematics. 
Among many important developments in graph limit theory—including the work by Backhausz, Kuehn et al.~\cite{Backhausz,Backhausz25} on spectral and dynamical aspects and the work by Hladký et al.~\cite{Hladky25,Hladky23,Hladky21} on combinatorial and probabilistic aspects—a Cheeger inequality has been established by Khetan and MJ~\cite{KM24}; however, deriving a sharper Cheeger-type inequality in the graph limit setting is much more difficult. 
Nevertheless, within our framework, 
Theorem \ref{thm:critical-inequality2}, directly yields a refinement of the Cheeger inequality for graph limits in terms of the first nontrivial eigenvalue of the nonlocal $p$-Laplacian; see Section~\ref{subsec:Cheeger-monotonicity}. 
\end{enumerate}

\subsection*{Acknowledgements}\phantomsection
\addcontentsline{toc}{section}{Acknowledgements}
Dong Zhang is supported by grants from  the  National Natural Science Foundation of China (No.\ 12401443). 
The idea for this paper first arose  at the Oberwolfach workshop \cite{Hein15} organised by Hein, Lenz and Mugnolo from 8 to 14 February 2015. After I had given my talk, Mugnolo asked the author: Can the graph 1-Laplacian characterize the bipartiteness of a graph? I subsequently came across Hein and Setzer’s paper \cite{Hein11} and, inspired by their work, obtained preliminary results for the set-pair Lov\'asz extension and used this tool to derive the signless 1-Laplacian whose smallest eigenvalue characterizes the graph bipartiteness; after joining the Max Planck Institute for Mathematics in the Sciences (MPI MiS) on 1 September 2019, I collaborated with Jost to further generalise the extension method; during this period, I also learnt Lov\'asz’s work on graph limits; finally, in the summer of 2024, I came across Lov\'asz’s paper \cite{Lovasz23}, which directly inspired me to extend my earlier thoughts to general measure spaces—and this led to the present paper. 
The author would therefore like to express his sincere gratitude  to J\"urgen Jost, Matthias Hein, Daniel Lenz, L\'aszl\'o Lov\'asz, Delio Mugnolo and Simon Setzer. 

\vspace{0.1cm}

\noindent\textbf{AI statement}. No AI tools were used in the research process and preparation of this manuscript. 
\section{Choquet-type Extensions
}\label{sec:disjoint-pair-extension}


{\normalsize Let $(J,\B)$ be  a sigma-algebra. Let $\varphi:\B\to \R$ be a set-function such that for any bounded measurable function $f:J\to\R$, the function $t\mapsto \varphi\{f\ge t\}$ is Lebesgue integrable on the interval $[\inf (f),\infty)$, where $\{f\ge t\}$  is shorthand for the upper  level set $\{x\in J:f(x)\ge t\}$, and $\inf (f):=\inf_{x\in J}f(x)$ denotes the infimum of $f$. 
The following integral 
\begin{equation}\label{eq:choquet}
\widehat{\varphi}(f):=\int_{\inf (f)}^\infty \varphi\{f\ge t\}dt+\varphi(J)\cdot\inf (f)    
\end{equation}
\noindent is called the \emph{Choquet extension} of $\varphi$ at $f$ (alternatively, it may be referred to as the \emph{Choquet integral} of $f$ with respect to $\varphi$, where $\varphi$ plays the role of a signed measure on $\B$). 
It is clear that $\widehat{\varphi}(1_A)=\varphi(A)$ for any $A\in\B$, where $1_A$ denotes the indicator function of $A$, i.e., $1_A(x)=1$, $\forall x\in A$, and $1_A(x)=0$, $\forall x\not\in A$. 
According to the results in \cite{Lovasz23,Lovasz25}, for a setfunction $\varphi$ with bounded variation and $\varphi(\emptyset)=0$,  its Choquet integral $\widehat{\varphi}$ is well-defined by the integral formula \eqref{eq:choquet}.} 


\subsection{Disjoint-pair Choquet extension}\label{sec:setpair-Choquet}

Before going into details of the disjoint-pair Choquet extension, we present a motivation for the study.  
As a useful generalization of matroids, the oriented matroid finds its applications in hyperplane arrangements, zonotope, Coxeter groups and many others \cite{orientedmatroid}. 
The rank function of oriented matroids is not submodular but bisubmodular \cite{Chandrasekaran}. 
This motivates us to investigate a generalized Choquet integral for bisubmodular functions, and it would be also useful to extend some parameters and certain results on (signed) graphs to signed submeasure spaces. %
This research idea stems primarily from research of submodular setfunctions—which can be traced back to the work of Choquet and Lov\'asz—as well as research into oriented matroids (or $\delta$-matroids),  bisubmodularity, and so on. 
Then, approaches in measure theory and functional analysis would be useful to investigate appropriate limit objects of some typical combinatorial structures \cite{BCLSV}. 
The goal of this paper is to derive meaningful identities and inequalities on bisubmodular and more general setpair functions. 
Below we propose a disjoint set-pair analog of Choquet integral. 
\begin{defn}[disjoint-pair Choquet extension]
Let 
$\psi:\B\times\B\to\R$ be a set-pair function such that for any bounded measurable function $f:J\to\R$, the function $t\mapsto \psi(\{f\ge t\},\{f\le-t\})$ is Lebesgue integrable on the interval $(0,\infty)$, where  $\{f\le-t\}$ is the lower level set of $f$ below $-t$. 
We call the following integral 
\begin{equation}\label{eq:choquet-pair}
\widehat{\psi}(f):=\int_0^\infty \psi(\{f\ge t\},\{f\le-t\})dt 
\end{equation}
the \emph{disjoint-pair Choquet extension} of $\psi$ at $f$ (alternatively, $\widehat{\psi}(f)$ is also referred to the \emph{disjoint-pair Choquet integral} of $f$ with respect to $\psi$). 
In the sequel, for convenience, we simply write $\psi(f\ge t,f\le -t)$ instead of $\psi(\{f\ge t\},\{f\le-t\})$.
\end{defn}

In many cases, the Choquet integral \eqref{eq:choquet} and its disjoint-pair analog \eqref{eq:choquet-pair} have useful relations. For example, given $\varphi$, if $J$ is a finite set, and if we take $\psi$ defined as  $\psi(A,A')=\varphi(A)+\varphi(J\setminus A')-\varphi(J)$ for any disjoint subsets $A$ and $A'$, then $\widehat{\psi}=\widehat{\varphi}$; see Proposition \ref{pro:setpair-generalize-original} for a slightly more general version
. In addition, if $\psi$ is a setpair function with bounded variation, then {\bf$\psi$ is bisubmodular\footnote{We note here that bisubmodular functions arise in a question of Lov\'asz \cite{Lovasz83} and in the context of delta-matroids \cite{Bouchet}. } if and only if $\widehat{\psi}$ is convex} (see Theorem \ref{thm:bisubmodular}). 

On the other hand, we observe that \eqref{eq:choquet-pair} provides a signed version of the Choquet integral, and it applies to certain variational problems on \emph{signed graphons}\footnote{The signed graphons, introduced in \cite{Lovasz-signed}, can be viewed as limits of signed graphs.}.  
This means, in some sense, the disjoint-pair Choquet integral  \eqref{eq:choquet-pair} enlarges the scope of the original one, which is useful for  
many problems on  signed graph limits and delta matroid limits. 

By formula \eqref{eq:choquet-pair}, to compute the disjoint-pair Choquet extension of a set-pair function $\psi$, we only use the $\psi$-value of setpairs $(B,B')\in \B\times \B$ with $B\cap B'=\varnothing$. Therefore, we introduce the sub-family $$\B_2:=\{(B,B')\in \B\times \B:B\cap B'=\varnothing\}$$ of $\B\times \B$. In fact, to define the Choquet-type extension $\widehat{\psi}$ via \eqref{eq:choquet-pair}, it suffices to have $\psi:\B_2\to\R$; and it is clear that $\widehat{\psi}(1_B-1_{B'})=\psi(B,B')$ for any $(B,B')\in\B_2$. For convenience, we use $\bm{\varPsi}(\B_2)$ to denote the collection of all $\psi:\B_2\to\R$ such that for any bounded measurable function $f:J\to\R$, the function $t\mapsto \psi(\{f\ge t\},\{f\le-t\})$ is Lebesgue integrable on the open interval $(0,\infty)$. 

\begin{lemma}
Given $\psi\in \bm{\varPsi}(\B_2)$, we have $\psi(\emptyset,\emptyset)=0$. Moreover, the map $\psi\mapsto \widehat{\psi}$ is linear, 
i.e., $\widehat{c\psi}=c\widehat{\psi}$, $\forall c\in\R$, $\forall \psi$;  and $\widehat{\psi_1+\psi_2}=\widehat{\psi_1}+\widehat{\psi_2}$, $\forall \psi_1,\psi_2$.
\end{lemma}
With the aid of the following definition, we can specify a sufficient condition for a function to belong to $\bm{\varPsi}(\B_2)$.
\begin{defn}
We say that a setpair function $\psi:\B_2\to\R$ has \emph{bounded variation} if there exists a constant $K\in\R$ such that 
$$\sum_{i=1}^n|\psi(X_i,Y_i)-\psi(X_{i-1},Y_{i-1})|\le K$$
for every chain of subset-pairs $(\emptyset,\emptyset)=(X_0,Y_0)\prec (X_1,Y_1)\prec\cdots\prec (X_n,Y_n)$. Here $(X_{i-1},Y_{i-1})\prec (X_i,Y_i)$ means that $X_{i-1}\subset X_i$ and $Y_{i-1}\subset Y_i$. 
\end{defn}

\begin{lemma}
\label{lem:BV-is-Psi}
Let $\psi:\B_2\to\R$ be a setpair function with bounded variation and with $\psi(\emptyset,\emptyset)=0$. 
Then, $\psi\in\bm{\varPsi}(\B_2)$ and the following properties hold:
\begin{enumerate}
\item[(a)] The functional $\widehat{\psi}:f\mapsto \widehat{\psi}(f)$ is positive homogeneous: if $f\in\Bd$ and $c>0$, then $\widehat{\psi}(cf)=c\widehat{\psi}(f)$, where $\Bd$ denotes the set of bounded measurable functions $f:J\to\R$.
\item[(b)] $\widehat{\psi}$ is Lipschitz continuous, i.e., there exists a constant $C>0$ depending on $\psi$ such that 
$|\widehat{\psi}(f)-\widehat{\psi}(g)|\le C\|f-g\|$ for any $f,g\in\Bd$, where $\|f-g\|:=\sup\limits_{x\in J}|f(x)-g(x)|$.
\end{enumerate}
\end{lemma}
\subsection{Bisubmodularity and Convexity}

Unlike submodularity which has both analytic and combinatorial lines of research, bisubmodularity only has the combinatorial line which goes back to Bouchet \cite{Bouchet}, Fujishige \cite{Fujishige} and Qi \cite{Qi88}, and has been extensively studied in discrete mathematics \cite{Huber14}. 
Historically, the absence of an analytic line is regrettable. 
Therefore, we generalize the concept of bisubmodularity from finite case to sigma-algebra case, not only for completing the analytic line, but also for future possibilities on graphons and limits of oriented matroids.

\begin{defn}[Bisubmodularity]
A set-pair function $\psi:\B\times\B\to\R$ (or $\psi:\B_2\to \R$) is \emph{bisubmodular} if for any  $(A,A'),(B,B')\in \B_2$,
\begin{equation}\label{eq:2-submodular}
\psi(A,A')+\psi(B,B')\ge \psi((A\cup B)\setminus (A'\cup B'),(A'\cup B')\setminus(A\cup B))+\psi(A\cap B,A'\cap B').
\end{equation}
\end{defn}
We also use the notions $\wedge$ and $\vee$ to express the intersection and union on $\B_2$, respectively.  They are defined as follows:  $(A,A')\wedge (B,B'):=(A\cap B,A'\cap B')$, and $(A,A')\vee (B,B'):=((A\cup B)\setminus (A'\cup B'),(A'\cup B')\setminus(A\cup B))$. Then the bisubmodularity \eqref{eq:2-submodular} can be rewritten as 
\begin{equation}
\label{eq:bisub2}
\psi(A,A')+\psi(B,B')\ge \psi ((A,A')\vee (B,B'))+\psi ((A,A')\wedge (B,B')).     
\end{equation}

Examples of bisubmodular functions include rank functions of  oriented matroids, bipartiteness functions on graphs, etc. 
Here, for simplicity, we give an easy bisubmodular function $\psi$ involving a graph $(V,E)$, defined via $\psi(S,S')=2\#E(S)+2\#E(S')+\#E(S\cup S',V\setminus(S\cup S'))$ in which $\#E(S)$ counts the number of edges in the subgraphs induced by $S$, and $\#E(S\cup S',V\setminus(S\cup S'))$ denotes the number of edges with exactly one vertex in $S\cup S'$. Such $\psi$ comes from the bipartiteness ratios on graphs. 
Then the setpair Choquet extension of $\psi$ is $\widehat{\psi}(f)=\sum_{\{u,v\}\in E}|f(u)+f(v)|$. 

\begin{defn}[Submodularity in components]
A set-pair function $\psi:\B\times\B\to\R$ is \emph{submodular in components} if 
$B\mapsto \psi(B,B')$ is a  submodular function of $B$; and 
$B'\mapsto \psi(B,B')$ is a  submodular function of $B'$. 
\end{defn}
It should be noted that submodularity in components does not imply bisubmodularity. 
We provide an example involving the homomorphism density:

\begin{example}\footnote{The author thanks Professor Lov\'asz for the interesting discussion and for raising the question of whether the homomorphism density is bisubmodular.}
For two simple graphs $F$ and $G$, it is known that the homomorphism density $t(F,G)$ is supermodular both as a function of the edge set of $F$ and as a function of the edge set of $G$ (see \cite{BCLSV06}). Suppose $\mathrm{v}(F)=\mathrm{v}(G)=n$, and let  $J={[n]\choose 2}
$ be all possible labeled edges on $n$ vertices. Let $\B$ be the power set of $J$. Then the homomorphism density $t(\cdot,\cdot)$ is 
a function $t:\B\times \B\to [0,1]$ satisfying certain properties. 
Graphs $F$ and $G$ can be viewed as two elements of $\B$. 
From the above discussion, $\psi(F,G):=1-t(F,G)$ defines a submodular function of $F\in\B$ for fixed $G$, and a submodular function of $G$ for fixed $F$. 
However, $\psi:=1-t$ is not bisubmodular. In fact, take two nonempty subsets $F,G\in\B$ with $F\cap G=\varnothing$.  
Then for such $(F,G)$ and $(G,F)$, we have $(F,G)\wedge(G,F)=(\varnothing,\varnothing)$ and $(F,G)\vee(G,F)=(\varnothing,\varnothing)$, and thus  $t(F,G)+t(G,F)>0=t((F,G)\wedge(G,F))+t((F,G)\vee(G,F))$, meaning that $\psi:=1-t$ does not satisfy \eqref{eq:bisub2}.

\end{example}

We further develop the following theorems characterizing the relation among bisubmodularity, bounded variation, and convexity.

\begin{theorem}
Let $\psi:\B_2\to\R$ be a bounded bisubmodular function. Then $\psi$ has  bounded variation.
\end{theorem}

\begin{theorem}\label{thm:bisubmodular}Let $\psi:\B_2\to\R$ be a setpair function with bounded variation and with $\psi(\emptyset,\emptyset)=0$. Then 
$\psi:\B\times\B\to\R$ (or $\psi:\B_2\to \R$) is bisubmodular if and only if $\widehat{\psi}$ is convex.
\end{theorem}


\section{Apply extension method to fractional optimization}\label{sec:frac}
Thanks to Choquet-type extensions, we establish an equivalence between fractional optimizations over set-variables and that over function-variables in this section. 
This can be applied to the cases of both graphons, Markov spaces, Riemannian manifolds, hypergraphs and metric measure spaces.

\subsection{Functionalization of set-based fractional optimization}
Recall the notation $\Bd$ used in Lemma \ref{lem:BV-is-Psi}, that is, the Banach space of bounded measurable functions on $(J,\B)$. 
Similar to $\bm{\varPsi}(\B_2)$, denote by $\bm{\varPsi}(\B)$ the collection of all setfunctions $\varphi:\B\to\R$ such that for any  $f\in\Bd$, the function $t\mapsto \varphi\{f\ge t\}$ is Lebesgue integrable on the interval $[\inf (f),\infty)$. 
To state our result in full generality, we introduce the following:
\begin{defn}
Given  a setfunction $\varphi\in\bm{\varPsi}(\B)$, a setpair function $\psi\in\bm{\varPsi}(\B_2)$, and two functionals $\widetilde{\varphi},\widetilde{\psi}:\Bd\to\R$.
\begin{itemize}
\item[\textbf{C1}] We say that $\widetilde{\varphi}$ is a \emph{sub-Choquet} (resp., \emph{super-Choquet}) \emph{extension} of $\varphi$, if $\widetilde{\varphi}(1_A)=\varphi(A)$ for any $A\in\B$, and $\widetilde{\varphi}(f)\le \widehat{\varphi}(f)$ (resp., $\widetilde{\varphi}(f)\ge \widehat{\varphi}(f)$) for any $f\in\Bd$.
\item[\textbf{C2}] We say that $\widetilde{\psi}$ is a \emph{sub-Choquet} (resp., \emph{super-Choquet}) \emph{extension} of $\psi$, if $\widetilde{\psi}(1_A-1_{A'})=\psi(A,A')$ for any $(A,A')\in\B_2$, and $\widetilde{\psi}(f)\le \widehat{\psi}(f)$ (resp., $\widetilde{\psi}(f)\ge \widehat{\psi}(f)$) for any $f\in\Bd$.
\item[\textbf{C3}] 
We say that $\widetilde{\varphi}$ is a \emph{one-homogeneous convex extension} of $\varphi$,  if $\widetilde{\varphi}$ is a lower semi-continuous  one-homogeneous convex even function, and  $\widetilde{\varphi}(1_A)=\varphi(A)$, $\forall A\in\B$. 
\item[\textbf{C4}] 
We say that $\widetilde{\psi}$ is a \emph{one-homogeneous convex extension} of $\psi$,  if $\widetilde{\psi}$ is a lower semi-continuous  one-homogeneous convex function, and  $\widetilde{\psi}(1_A-1_{A'})=\psi(A,A')$  for any $(A,A')\in\B_2$. 
\end{itemize}    
\end{defn}

\begin{theorem}\label{thm:tilde-fg-equal}
Given $\varphi_1,\varphi_2\in\bm{\varPsi}(\B)$ which are nonnegative and satisfy $\varphi_1(J)=\varphi_2(J)=0$, we have 
\begin{equation}\label{eq:Choquet-identity}
\inf_{A\in\B\text{ with }\varphi_2(A)>0}\frac{\varphi_1(A)}{\varphi_2(A)} =\inf_{f\in \Bd\text{ with }\widehat{\varphi}_2(f)>0}\frac{\widehat{\varphi}_1(f)}{\widehat{\varphi}_2(f)}   =\inf_{f\in \Bd\text{ with }\widetilde{\varphi}_2(f)>0}\frac{\widetilde{\varphi}_1(f)}{\widetilde{\varphi}_2(f)} .  
\end{equation}
where $\widetilde{\varphi}_1$ is a super-Choquet extension  of $\varphi_1$, and $\widetilde{\varphi}_2$ is a sub-Choquet extension  of $\varphi_2$ (or a one-homogeneous convex extension of $\varphi_2$ if $\varphi_2$ is submodular). 

Let $\psi_1, \psi_2 \in \bm{\varPsi}(\B_2)$, and assume that they are both nonnegative. 
Then \begin{equation}\label{eq:Choquet-pair-identity}
\inf_{(A,A')\in \B_2\text{ and }\psi_2(A,A')>0}\frac{\psi_1(A,A')}{\psi_2(A,A')}=\inf_{f\in \Bd\text{ with }\widehat{\psi}_2(f)>0}\frac{\widehat{\psi}_1(f)}{\widehat{\psi}_2(f)}=\inf_{f\in \Bd\text{ with }\widetilde{\psi}_2(f)>0}\frac{\widetilde{\psi}_1(f)}{\widetilde{\psi}_2(f)}   
\end{equation}
where $\widetilde{\psi}_1$ is a super-Choquet extension  of $\psi_1$, and $\widetilde{\psi}_2$ is a sub-Choquet extension of $\psi_2$  (or a one-homogeneous convex extension of $\psi_2$ if $\psi_2$ is bisubmodular).

\end{theorem}

\begin{remark}
For such a measurable space $(J,\B)$, we usually assume that there is a Borel measure $\mu$, so that  $(J,\B,\mu)$ becomes a Borel measure space. In addition, unless otherwise stated, we always assume $0<\mu(J)<\infty$. 
Then, we can also work on $L^\infty(J) :=\{f:J\to\R: \exists \widetilde{f}\in\Bd\text{ s.t. }\mu\{x\in J:f(x)\ne \widetilde{f}(x)\}=0\}$ instead of $\Bd$, and use $\essinf$ instead of $\inf$. 
However, for the sake of brevity, we intend to work on $\Bd$ rather than $L^\infty$. 
Moreover, we present at the end of Section \ref{sec:proof-main-equal} a formal `sup'-analog of Theorem \ref{thm:tilde-fg-equal}—that is, by  replacing all `inf' with `sup' in \eqref{eq:Choquet-pair-identity} and \eqref{eq:Choquet-identity}; see Theorem \ref{th:dual-main}.  
We also would like to point out that $\widehat{\psi}_2(f)\ge0$ for any $f\in\Bd$, while  $\widehat{\varphi}_2(f)\ge0$ can determine a closed cone that is a proper subset of $\Bd$. Thus, compared with the extension of a single set function, the extension of set-pair functions yields a Rayleigh functional that is defined globally, which facilitates the direct use of variational methods. In contrast, the Rayleigh functional obtained from the extension of a setfunction typically corresponds to an optimization problem restricted on a convex cone (the first orthant), subject to cone constraints, and hence is not convenient for direct variational treatment. 
\end{remark}

\noindent{\bf Applicability of Theorem \ref{thm:tilde-fg-equal}}: 
Theorem \ref{thm:tilde-fg-equal} establishes a framework based on a Choquet-type extension, which provides a method for equivalently transforming fractional optimization over set-variables to that over function-variables, thereby facilitating the subsequent application of the calculus of variations. 
This is applicable for many distinct spaces and different problems, such as maxcut, bipartiteness ratio and conductance on graph limits, as well as Cheeger constant and torsional rigidity on Riemannian manifolds, etc. 


We apply Theorem \ref{thm:tilde-fg-equal} to obtain equivalent explicit formulations for some quantities on graphons and manifolds; see Section \ref{subsec:first-example} for some typical examples and applications of Theorem \ref{thm:tilde-fg-equal}. Combining these equivalent formulations, and monotonicity inequalities 
established in Theorems \ref{thm:critical-inequality} and \ref{thm:critical-inequality2}, we provide sharp bounds for many important parameters on various spaces; see Sections \ref{sec:Lp-integration} and \ref{sec:Applications}.

\subsection{Applications on graphons and Riemannian manifolds}\label{subsec:first-example}

In this section, we apply Theorem \ref{thm:tilde-fg-equal} to give equivalent representations for many combinatorial and geometric quantities. 
Since at least half of these quantities correspond to the concept of the graphon, we shall briefly outline its definition and background. 
Generally speaking, consider a pair $(J,W)$ in which $J$ is a measure space with finite measure, and $W:J\times J\to I$ is a symmetric measurable function, i.e., $W(y,x)=W(x,y)$ for any $(x,y)\in J\times J$, where $I$ is commonly set to be the interval $[0,1]$. In the literature, such a $W$ is referred to as a \emph{graphon}; see \cite{Lovasz-book}. 
We call $W$ a bounded graphon, if $I$ is replaced by a bounded interval $[0,N]\subset [0,\infty)$ for some $N>0$. 

\subsubsection{Maxcut on graphon}
The first example we would like to give is the maxcut problem on graphon. 
The graphon provides a general, elegant criterion for testability: a parameter is testable if and only if it is continuous with respect to the cut distance on the space of graphons. In this framework, the MaxCut parameter is a prime example \cite{Borgs06}. It has been proven to be continuous in this metric, which provides a short, unified proof of its testability. 
Below, we provide a functional reformulation of the maxcut problem on graphons, which, together with the monotonicity inequality established in Section \ref{sec:Lp-integration}, shows hidden connections to other important graphon parameters. 
\begin{example}[Maxcut on graphon]\label{exam:maxcut}
Given a graphon $W:J\times J\to [0,1]$, for each Borel subset $S\subset J$, let $$\partial S=\{(x,y)\in J\times J:x\in S,y\not\in S\text{ or }y\in S,x\not\in S\}$$
denote the edge-boundary of $S$, and let $\eta(\partial S)=\int_{\partial S}W(x,y)dxdy$ be the measure of $\partial S$. The cut-size of $S$ is defined as $\mathrm{cut}(S):= \eta(\partial S)$. Then 
$$\mathrm{cut}(S)=2\int_{S}\int_{J\setminus S}W(x,y)dxdy.$$ Taking $\psi_1(S,S')=\mathrm{cut}(S)+\mathrm{cut}(S')$ and $\psi_2(S,S')=2$, $\forall (S,S')\in\B_2\setminus\{(\emptyset,\emptyset)\}$, we can easily obtain  $\widehat{\psi}_1(f)=\int_{J}\int_{J} W(x,y)|f(x)-f(y)|dxdy$ which refers to the total variation of $f$, and we can derive $\widehat{\psi}_2(f)=2\|f\|_\infty$. Then by Theorem \ref{thm:tilde-fg-equal}, we have
$$ \text{MaxCut}(W):=\sup_{S\in\B}\mathrm{cut}(S)=\sup_{f\in \Bd }\frac{\int_{J}\int_{J} W(x,y)|f(x)-f(y)|dxdy}{2\|f\|_\infty}$$
which indicates the \emph{maxcut} value of the graphon $W$; see \cite{Zhang/graphon} for details on the maxcut problem on graphons. 
\end{example}

We remark that the cut function $\mathrm{cut}(\cdot)$ is similar to the cut-capacity function in the Max-Flow-Min-Cut theorem \cite{Lovasz21}.
\subsubsection{Bipartiteness ratio on graphon}
The second example is the bipartiteness ratio on graphons, which is a quantitative version of the bipartiteness of a graphon. 
The dual Cheeger inequality that links the bipartiteness ratio and the spectrum of the graphon Laplacian is also an example of the powerful monotonicity inequality (Theorem \ref{thm:critical-inequality2} in Section \ref{sec:Lp-integration}), from which, one can immediately see that a graphon is exactly bipartite (i.e., it admits a measurable partition into two independent sets) if and only if its spectrum is symmetric about zero. 
We provide in the following example an equivalent reformulation of the bipartiteness ratio on a graphon, which helps to understand some new spectral bounds for bipartiteness ratio presented in Section \ref{subsec:spectral-bound}.
\begin{example}[Bipartiteness ratio on graphon]\label{exam:bipartiteness} For two Borel subsets $S,S'\in\B$, 
denote by $e(S,S')=\int_S\int_{S'}W(x,y)dxdy+\int_{S'}\int_{S}W(x,y)dxdy$. Let $\psi_1(S,S')=e(S,S')$ and $\psi_2(S,S')=\vol(S)+\vol(S')$, where $\vol(S):=\int_{S}\deg(x)dx$ and $\deg(x):=\int_J W(x,y)dy$. 
Then $$\widehat{\psi}_1(f)=\int_J\deg(x)|f(x)|dx-\frac12\int_{J}\int_{J} W(x,y)|f(x)+f(y)|dxdy$$ and $\widehat{\psi}_2(f)=\int_J\deg(x)|f(x)|dx$. 
We can then apply Theorem \ref{thm:tilde-fg-equal} to derive
$$\beta_W:= 1-\sup_{S,S'\in\B,S\cap S'=\varnothing,\vol(S\cup S')>0}\frac{e(S,S')}{\vol(S\cup S')}=\inf_{f\in \Bd}\frac{\int_J\int_JW(x,y)|f(x)+f(y)|dxdy}{2\int_J\deg(x)|f(x)|dx} $$
which refers to the \emph{bipartiteness ratio} of the graphon $W$. This concept follows Trevisan's work on maxcut problem for graphs \cite{Trevisan2012}, in which he utilised the graph bipartiteness ratio to design a recursive algorithm for solving the maxcut problem. 
Note also that $1-\beta_W$ provides an upper bound for the maxcut value $\text{MaxCut}(W)$ via the inequality $\text{MaxCut}(W)\le (1-\beta_W)\|W\|_1$, see Corollary \ref{cor:bipart-maxcut}. 
\end{example}

\subsubsection{Conductance on graphon}
The third example is the conductance (or Cheeger constant) on graphons, which quantifies the connectedness of a graphon, and which also satisfies a Cheeger inequality almost the same to the graph case \cite{KM24}. 
We show an equivalent functional representation of the graphon conductance in the following example, and this functional representation, again directly combining with the monotonicity inequality in Theorem \ref{thm:critical-inequality2}, will be used to derive refined Cheeger-type inequalities for the graphon conductance. 
\begin{example}[Conductance on graphon]\label{example:conductance-graphon}
Given a graphon $W:J\times J\to [0,1]$, for each Borel subset $S\subset J$, let $\varphi_1(S)=\mathrm{cut}(S)
$, which will be regarded as the measure of the boundary of $S$. Let $\varphi_2(S)=\min\{\vol(S),\vol(J\setminus S)\}
$. 
Then, by elementary computation, we have $\widehat{\varphi}_1(f)=\int_{J}\int_{J} W(x,y)|f(x)-f(y)|dxdy$ and $\widehat{\varphi}_2(f)=\min\limits_{t\in\R}\int_{J} \deg(x)|f(x)-t|dx$. Consequently, Theorem \ref{thm:tilde-fg-equal} implies 
$$ \inf_{S\in\B:\vol(S)\vol(J\setminus S)>0}\frac{\mathrm{cut}(S)}{\min\{\vol(S),\vol(J\setminus S)\}}=\inf_{\text{nonconstant }f\in \Bd}\frac{\int_J\int_JW(x,y)|f(x)-f(y)|dxdy}{\min\limits_{t\in\R}\int_{J} \deg(x)|f(x)-t|dx} .$$
The quantity in the left hand side is called the \emph{conductance} of the graphon $W$, 
and with careful construction of $\widetilde{\varphi}_1$ (or $\psi_1$ and $\psi_2$), it can also be written as the following alternative functional reformulation
$$ \inf_{\text{nonconstant }f\in \Bd}\frac{2\|W\|_1\|f\|_\infty-\int_J\int_JW(x,y)|f(x)+f(y)|dxdy}{\min\limits_{t\in\R}\int_{J} \deg(x)|f(x)-t|dx} $$
which, however, possesses some advantages on its corresponding Euler-Lagrange equation, particularly on the computation since the minimizers of the nonlinear Rayleigh quotient $(\|\cdot\|_\infty-\langle g,\cdot\rangle)/\|\cdot\|_p$ can be expressed in closed form. 
\end{example}

\subsubsection{Frustration on signed graphons}
As a key measure for analyzing signed networks, the frustration index on a signed graph quantifies how far it is from being balanced.  
This concept can be easily extended to  signed graphons, which quantifies how far a signed graphon $W$ is from its underlying graphon $|W|$. Here, a signed graphon is defined as  a measurable symmetric function $W: J\times J\to [-1,1]$, and the underlying graphon $|W|$ is the nonnegative measurable symmetric function simply defined via $|W|(x,y):=|W(x,y)|$, $\forall (x,y)\in J\times J$.

\begin{example}
\label{example:frustration}
We focus on the \emph{frustration index} of a signed graphon $W$, namely,
$$\mathrm{Fru}(W):=\min_{\text{measurable }f:J\to \{-1,1\}}\int_{J\times J}|W(x,y)|\big|f(x)-\mathrm{sgn}(W(x,y))f(y)\big|dxdy.$$
Then, \begin{align*}
 \mathrm{Fru}(W)&=\inf_{A\in \B}4\eta(W_+[A\times A^c])+2\eta \big(W_-[(A\times A)\cup(A^c\times A^c)]\big)   
 \\&=2\eta(W_-[J\times J])+4\inf_{A\in \B}\big(\eta(W_+[A\times A^c])-\eta(W_-[A\times A^c])\big)
\end{align*}
where $W_\pm[S]:=\{(x,y)\in S:\pm W(x,y)>0\}$ and   $\eta(S):=\int_S|W(x,y)|dxdy$ for any measurable $S\subset J\times J$. Theorem \ref{thm:tilde-fg-equal} yields that the frustration index satisfies the following identity
$$ \mathrm{Fru}(W)=2\|W_-\|_1+\inf\limits_{f\ne0}\frac{\int_{J\times J}W(x,y)|f(x)-f(y)|dxdy}{\|f\|_\infty}$$
where $W_-(x,y):=\min\{W(x,y),0\}$, $\forall (x,y)\in J\times J$.
\end{example}

\subsubsection{Dirichlet $p$-isoperimetric constant on Riemannian manifolds}

Suppose that $M$ is a Riemannian manifold of dimension $n$. 
Let $\Omega\subset M$ be an open bounded domain with Lipschitz boundary. The  Dirichlet $p$-isoperimetric constant of $\Omega$ is defined by 
$$\mathrm{ID}_p(\Omega):=\inf_{A\subset \Omega:\,\partial A\cap\partial \Omega=\varnothing}\frac{\mathcal{H}^{n-1}(\partial A)}{\big(\mathcal{H}^{n}( A)\big)^{\frac1p}}$$
where $\mathcal{H}^{j}$ indicates the standard $j$-dimensional Hausdorff measure for $j=n-1$ or $j=n$.  

Let $\varphi_1(A)=\mathcal{H}^{n-1}(\partial A)$ for any $A\in\B$. Then $\widehat{\varphi}_1(f)=\int_\Omega|\nabla f(x)|dx$. 

Let $\widetilde{\varphi}_2(f)=\|f\|_{L^p(\Omega)}$. Then $\widetilde{\varphi}_2$ is a one-homogeneous convex even function, and $\widetilde{\varphi}_2(1_A)=\|1_A\|_{L^p(\Omega)}=\big(\mathcal{H}^{n}(A)\big)^{\frac1p}$. Therefore, Theorem \ref{thm:tilde-fg-equal} yields the following equality 
$$\mathrm{ID}_p(\Omega)=\inf_{f\in  W^{1,1}\cap L^p:\,f|_{\partial \Omega}=0}\frac{\int_\Omega|\nabla f(x)|dx}{\|f\|_{L^p(\Omega)}}.$$

\subsubsection{Dirichlet $p$-isoperimetric constant on graphons}

As a generalization of the result in the previous section, we introduce the Dirichlet $p$-isoperimetric constant on graphons.

\begin{example}

Fixed $J_0\in \B$ with $0<\mu(J_0)<\mu(J)$ 
and let $\varphi_1(S)=\mathrm{cut}(S)$ for $S\in\B_{J_0}$, where $\B_{J_0}:=\{S\in \B:S\subset J_0\}$. Given $p\ge 1$, let  $\varphi_2(S)=\vol(S)^{\frac1p}$.  
Then  by Theorem \ref{thm:tilde-fg-equal}, the Dirichlet $p$-isoperimetric constant on $(J,W)$ is 
$$\mathrm{ID}_{p,J_0}:=\inf_{S\in\B_{J_0}}\frac{\mathrm{cut}(S)}{\vol(S)^{\frac1p}}=\inf_{\text{nonzero }f\in \Bd\text{ with }f|_{J\setminus J_0}=0}\frac{\int_J\int_JW(x,y)|f(x)-f(y)|dxdy}{\big(\int_{J} \deg(x)|f(x)|^pdx\big)^{\frac1p}} . $$
\end{example}
We summarise the conclusions from the above examples as follows.
\begin{corollary}\label{cor:6-equality}The following are the equivalent explicit expressions for maxcut, bipartiteness ratio, conductance, frustration index, and Dirichlet $p$-isoperimetric constant, respectively.  
\begin{enumerate}
\item The maxcut on graphon
$$ \mathrm{MaxCut}(W)=\sup_{\text{nonzero }f\in \Bd }\frac{\int_{J}\int_{J} W(x,y)|f(x)-f(y)|dxdy}{2\|f\|_\infty}$$
\item The bipartiteness ratio on graphon
$$\beta_W=\inf_{\text{nonzero }f\in \Bd}\frac{\int_J\int_JW(x,y)|f(x)+f(y)|dxdy}{2\int_J\deg(x)|f(x)|dx} $$
\item The conductance on graphon
\begin{align*}
h_W&=\inf_{\text{nonconstant }f\in \Bd}\frac{\int_J\int_JW(x,y)|f(x)-f(y)|dxdy}{\min\limits_{t\in\R}\int_{J} \deg(x)|f(x)-t|dx}
\\&=
\inf_{\text{nonconstant }f\in \Bd}\frac{2\|W\|_1\|f\|_\infty-\int_J\int_JW(x,y)|f(x)+f(y)|dxdy}{\min\limits_{t\in\R}\int_{J} \deg(x)|f(x)-t|dx} 
\end{align*}
\item The Dirichlet $p$-isoperimetric constant on graphon
$$\mathrm{ID}_{p,J_0}=\inf_{\text{nonzero }f\in \Bd\text{ with }f|_{J\setminus J_0}=0}\frac{\int_J\int_JW(x,y)|f(x)-f(y)|dxdy}{\big(\int_{J} \deg(x)|f(x)|^pdx\big)^{\frac1p}} $$
\item The frustration index on signed graphon
$$ \mathrm{Fru}(W)=\inf\limits_{\text{nonzero }f\in \Bd}\frac{\int_{J\times J}W(x,y)|f(x)-f(y)|dxdy}{\|f\|_\infty}+2\int_{J\times J}|\min\{W(x,y),0\}|dxdy$$
\item Dirichlet $p$-isoperimetric constant of bounded domains in Riemannian manifolds:
$$\mathrm{ID}_p(\Omega)=\inf_{f\in  W^{1,1}\cap L^p:\,f|_{\partial \Omega}=0}\frac{\int_\Omega|\nabla f(x)|dx}{\|f\|_{L^p(\Omega)}}.$$
\end{enumerate}
\end{corollary}

These equivalent representations serve two main purposes: firstly, they can be combined with the monotonicity theorem proposed in Section \ref{sec:Lp-integration} to derive many useful inequalities, see Section \ref{sec:Applications}; secondly, these explicit functional representations in closed form can be directly used to design algorithms for these combinatorial problems. 

There are sharp spectral bounds for the maxcut value, the bipartiteness ratio, and the conductance on graphons, which will be discussed in detail in Sections \ref{sec:Lp-integration} and \ref{sec:Applications}. 




\section{
$L^p$ integration of a family 
of Choquet extensions 
}
\label{sec:Lp-integration}


In this section, we propose $L^p$ integration of a family of Choquet integrals (Definitions \ref{def:Phi-integration} and \ref{def:Psi-integration}), which allows us to apply the extension approaches to estimate many combinatorial and geometric quantities. 

\subsection{min-max extension constants}
\begin{defn}[$L^p$ integration of Choquet extensions]\label{def:Phi-integration}
Given a measure space $E$ with a measure $\eta$, and a family of setfunctions 
$$\varphi_e:\B \to [0,+\infty)\text{ with }\varphi_e(J)=\varphi_e(\emptyset)=0,\; \forall e\in E,$$
for $p\in[1,\infty)$, we define $$\Phi_p(f):=\Big(\int_E \big(\widehat{\varphi}_e(f)\big)^pd\eta(e)\Big)^{\frac1p},$$
and 
$$\big\llbracket f \big\rrbracket_p :=\Big(\int_E|f|_{\varphi_e}^pd\eta(e)\Big)^{\frac1p}$$
where $|f|_{\varphi_e}:=\inf \big\{t>0: \varphi_e\{f\ge s\}=\varphi_e\{f\ge -s\}=0,\forall s\ge t\big\}$ represents the intrinsic infinity norm of $f$ with respect to $\varphi_e$.
\end{defn}

Replacing the original Choquet extension used in Definition \ref{def:Phi-integration} with the disjoint-pair Choquet extension, we have the following definition. 
\begin{defn}[$L^p$ integration of disjoint-pair Choquet extensions]\label{def:Psi-integration}
Let $E$ be a measure space equipped with a measure $\eta$. Given a  family of set-pair functions 
$$\psi_e:\B\times \B\to [0,+\infty)\text{ with }\psi_e(\emptyset,\emptyset)=0,\;\forall e\in E,$$
for $p,q\ge 1$, we define
$$\Psi_p(f):=\Big(\int_E \big(\widehat{\psi}_e(f)\big)^pd\eta(e)\Big)^{\frac1p},$$
and 
$$\big\llbracket f\big\rrbracket_q:=\Big(\int_E\|f\|_{\psi_e}^qd\eta(e)\Big)^{\frac1q}$$
where $\|f\|_{\psi_e}=\inf\{t>0:\psi_e(f\ge s,f\le-s)= 0,\forall s\ge t\}$  denotes  the intrinsic infinity norm of $f$ with respect to $\psi_e$. 
\end{defn}

It is clear that for any $p,q\in[1,\infty]$, $\Phi_p$, $\Psi_p$ and $\llbracket \cdot\rrbracket_q$ are  positively one-homogeneous functionals. Based on Definitions \ref{def:Phi-integration} and \ref{def:Psi-integration}, we further propose the following definition: 
\begin{defn}[min-max extension constants]
\label{def:global-extension-cons}
A nonempty set family $\mathcal{Y}\subset 2^\Bd$ is \emph{admissible} if $\mathcal{Y}$ is invariant under Mazur-type maps, that is, if $Y\in \mathcal{Y}$, then for any $r>0$, $\{f^r:f\in Y\}\in \mathcal{Y}$, where $f\mapsto f^r$ refers to the Mazur map of order $r$ (see \eqref{eq:Mazur-map} for the definition). 
We define the 
\emph{min-max value} of $\Phi_p/\llbracket \cdot\rrbracket_q$ with respect to an admissible $\mathcal{Y}$ as 
$$ 
c
(\Phi_p,\llbracket \cdot\rrbracket_q):=\inf\limits_{Y\in \mathcal{Y}}\sup\limits_{f\in Y} \frac{\Phi_p(f)}{\big\llbracket f\big\rrbracket_q}.$$
Similarly, 
$$ 
c
(\Psi_p,\llbracket \cdot\rrbracket_q):=\inf\limits_{Y\in \mathcal{Y}}\sup\limits_{f\in Y} \frac{\Psi_p(f)}{\big\llbracket f\big\rrbracket_q}.$$
We 
call $c
(\Phi_p,\llbracket \cdot\rrbracket_q)$ and $c
(\Psi_p,\llbracket \cdot\rrbracket_q)$ the \emph{min-max extension constants}.

\end{defn}

Now we show a concrete construction of the admissible set family $\mathcal{Y}$. 
For any $k\in \mathbb{N}_+$, let 
\begin{equation}\label{eq:Y_k}
\mathcal{Y}_k=\{A\subset \Bd\setminus\{0\}:\mathrm{genus}(A)\ge k\}    
\end{equation}
where $\mathrm{genus}(A)$ is the Krasnoselskii genus of $A$ defined by 
\[
\mathrm{genus}(A)=
\inf\big\{ k\in\mathbb{N}\;|\; \exists\,\text{continuous }\zeta:A\to \mathbb{S}^{k-1}\;\, s.t. \;\, \zeta(x)=-\zeta(-x)\big\rbrace  
\]
and $\mathbb{S}^{k-1}$ is the standard unit sphere of $\R^k$. Then, it is known that $\mathcal{Y}_k$ is admissible. 
The min-max extension constants with respect to such $\mathcal{Y}_k$ will be simply written as $c_k
(\Phi_p,\llbracket \cdot\rrbracket_q)$ and $c_k
(\Psi_p,\llbracket \cdot\rrbracket_q)$. 
Furthermore, the admissible set family $\mathcal{Y}$ can also be taken as the single set $\{\Bd\setminus\{0\}\}$.

The 
min-max extension constants introduced in the above definition can be applied to express lots of useful combinatorial and analytic parameters, which we will explain in Section \ref{sec:Applications}. 

\subsection{Monotonicity of various quantities on measure spaces
}

Our main result is the following monotonicity inequalities which integrate many useful combinatorial parameters and analytic quantities. 
\begin{theorem}[Extended Monotonicity Inequality]\label{thm:critical-inequality}
Given $p,q,s,t\ge1$ with $ps'=qt'$, where $s'$ and $t'$ are the H\"older conjugates of $s$ and $t$, respectively.  
We have the following monotonicity inequality 
\begin{equation}\label{eq:c-monotone1}
c(\Phi_p,\llbracket \cdot\rrbracket_{q})\le t\cdot  c(\Phi_{ps},\llbracket \cdot\rrbracket_{qt}).    
\end{equation}
In particular, $p\cdot c(\Phi_p,\llbracket \cdot\rrbracket_{p})$ is increasing with respect to $p$. In addition, $p\mapsto \eta(E)^{-1/p}c(\Phi_p,\llbracket \cdot\rrbracket_{q})$ increases, and $q\mapsto \eta(E)^{1/q}c(\Phi_p,\llbracket \cdot\rrbracket_{q})$ decreases.  

Moreover, all the above properties still hold when we replace $\Phi$ by $\Psi$, and replace $\llbracket \cdot\rrbracket$ by its corresponding set-pair  version. 
\end{theorem}

\begin{remark}
We can rephrase \eqref{eq:c-monotone1} as follows: 
for any fixed $q\ge 1$, take $p=\theta q\ge 1$; if $0<\theta\le1$, then it follows from \eqref{eq:c-monotone1} that the function $$[1,\infty)\ni t\mapsto t\cdot c(\Phi_{\frac{q\theta t}{\theta+(1-\theta)t}},\llbracket \cdot\rrbracket_{qt})$$ is nondecreasing; similarly, if $1<\theta<\infty$, then $$[1,\frac{\theta}{\theta-1})\ni t\mapsto t\cdot c(\Phi_{\frac{q\theta t}{\theta-(\theta-1)t}},\llbracket \cdot\rrbracket_{qt})$$ is nondecreasing.  In particular, if $\theta=1$, i.e., $q=p$, then  $[1,\infty)\ni t\mapsto t\cdot 2^{\frac{1}{pt}}c(\Phi_{pt},\llbracket \cdot\rrbracket_{pt}) $  is nondecreasing. 
Therefore, \eqref{eq:c-monotone1} in Theorem \ref{thm:critical-inequality} is essentially a monotonicity inequality. 
\end{remark}

Theorem \ref{thm:critical-inequality} encompasses a number of inequalities—torsional rigidity, packing radius, relatively isoperimetric, bipartiteness ratio, and signed conductance—each of which is valid on a broad class of underlying structures, such as graphs, hypergraphs, graphons, graphings, Markov spaces, other graph limits, Euclidean domains, Riemannian manifolds, and metric measure spaces.


\vspace{0.19cm}

Note that Definitions \ref{def:Phi-integration}, \ref{def:Psi-integration} and \ref{def:global-extension-cons} involve two measure spaces $(J,\mu)$ and $(E,\eta)$. 
It is worth noting that if $J$ and $E$ satisfy certain conditions, then Theorem \ref{thm:critical-inequality}  can be strengthened. 
To this end, we first propose the following finitely related structure linking two measure spaces.
\begin{defn}[Finitely related structure]\label{def:hyper-degree-related}
We say that two measurable spaces $J$ and $E$ form a \emph{finitely related structure} if there exists a correspondence $E\rightrightarrows J$ such that every $e\in E$ corresponds to a finite number of elements in $J$. 
Denote by $e_J:=\{x\in J:e\text{ corresponds to }x\}$. 

In addition, suppose that $J$ and $E$ form a finitely related structure. Let $\mu$ and $\eta$ be measures on $J$ and $E$, respectively. Assume there exists a degree function $\deg:J\to[0,+\infty)$ such that for any test function $g\in\Bd$,
\begin{equation}\label{eq:degree-related}
\int_E \sum_{x\in e_J} g(x)d\eta(e)=\int_J\deg(x) g(x)d\mu(x) .   
\end{equation}
Then we say that $J$ and $E$ are \emph{degree-related}.
\end{defn} 

Intuitively, if $J$ and $E$ form a finitely related structure, then a typical degree function can be defined as $\overline{\deg}(x)=\overline{\eta}\{e\in E:e_J\ni x\}$, where $\overline{\eta}$ is a certain Borel measure on $E$ such that \eqref{eq:degree-related} holds when taking $\deg=\overline{\deg}$ and $\eta=\overline{\eta}$. 
Below, we present four simple examples that satisfy the properties introduced in Definition \ref{def:hyper-degree-related}.
\begin{example}[hypergraph \cite{Berge}]
For a standard hypergraph $(V,E)$ with vertex set $V$ and hyperedge set $E$, the measures $\eta$ and $\mu$ are the counting measures on $E$ and $V$, respectively. The standard degree is defined as $\deg(x)=\#\{e\in E:x\in e\}$. 
\end{example}

\begin{example}[graphon \cite{Lovasz-book}]
Let $W:J\times J\to [0,1]$ be a graphon. Let $E=J\times J$ be equipped with a measure $\eta$ defined by $\eta(S):=\int_{S}W(x,y)dxdy$ for any Borel subset $S$ of $J\times J$. 
And the measure $\mu$ is the standard Lebesgue measure. 
In this case, one can use the 
degree function $\deg(x)=2\int_J W(x,y)dy$ for any $x\in J$. 
\end{example}

\begin{example}[graphing \cite{Lovasz-book}]
Let $(J,\lambda,E,\deg)$ be a graphing. For any Borel subsets $A,B\subset J$, $\eta(A\times B)=\int_{B}\deg_A(x)d\lambda(x)$.  
Then, the commonly used degree defined by $\deg(x)=\#\{e\in E:x\in e\}$ for any $x\in J$, and the induced measure $\mu$ defined by $\mu(A)=\eta(A\times J)$ for any $A\in\B$, satisfy \eqref{eq:degree-related}.  
\end{example}

\begin{example}[Markov space \cite{KLS24}]Let $(J\times J,\B\times\B,\eta)$ be a Markov space. Taking $E=J\times J$, let $\mu$ be the marginal measure of $\eta$, i.e., $\mu(A)=\eta(A\times J)=\eta(J\times A)$ for any $A\in\B$. In this case, $\int_{J\times J}(g(x)+g(y))d\eta(x,y)=\int_J 2g(x)d\mu(x)$ meaning that $\deg(x)=2$ for $\mu$-a.e. $x\in J$. 
We refer to \cite{KLS24} for the precise definition of Markov spaces.
\end{example}

We further propose the following definition:

\begin{defn}\label{def:hyper-phi-concentrate}
Suppose that $J$ and $E$ form a finitely related structure.  
Given $e\in E$, we say that $\varphi_e$ is \emph{concentrated} on $e$ if $\varphi_e(B)=\varphi_e(B\cap e_J)$ for all $B\in\B$;  
we say that $\psi_e$ is \emph{concentrated} on $e$ if $\psi_e(B,B')=\psi_e(B\cap e_J,B'\cap e_J)$ for all $(B,B')\in\B_2$. 
\end{defn}

We note that Markov spaces fully   satisfy the above degree-related and concentration conditions (Definitions \ref{def:hyper-degree-related} and 
\ref{def:hyper-phi-concentrate}), but manifolds do not. 
Therefore, notably, Definitions \ref{def:hyper-degree-related} and \ref{def:hyper-phi-concentrate} also give a fundamental difference between graph limits and manifolds. 

One can then consider the following $L^p$-norm defined in the form of $$\widetilde{\big\llbracket   f\big\rrbracket}_p=\Big(\int_J\widetilde{\deg}(x)|f(x)|^pd\mu(x)\Big)^{\frac1p}$$ 
with $\widetilde{\deg}:J\to \R$ being an admissible degree function satisfying the conditions in Definition \ref{def:hyper-degree-related}. 
Moreover, we have the min-max extension constants $c(\Phi_p,\widetilde{\llbracket \cdot\rrbracket}_{q})$ which satisfy the following stronger extended monotonicity inequality.
\begin{theorem}[Extended Monotonicity Inequality]\label{thm:critical-inequality2}
Suppose that $J$ and $E$ are degree-related with $\#e_J\le 2$ for any $e\in E$, and $\varphi_e$ is concentrated on $e$, as well as $C_\Phi:=\sup_{e\in E}\|\varphi_e\|_\infty<\infty$.  
Then, $$r\mapsto \big(\frac{1}{2C_\Phi}c(\Phi_{pr},\widetilde{\llbracket \cdot\rrbracket}_{qr})\big)^r$$ is decreasing with respect to $r\ge 1$. Furthermore, \eqref{eq:c-monotone1} in Theorem \ref{thm:critical-inequality} can be strengthened as
$$(2C_\Phi)^{1-t}c^t(\Phi_{pt},\widetilde{\llbracket \cdot\rrbracket}_{qt})\le c(\Phi_p,\widetilde{\llbracket \cdot\rrbracket}_{q})\le t\cdot  c(\Phi_{ps},\widetilde{\llbracket \cdot\rrbracket}_{qt}) $$
and this inequality still holds when we replace condition $\#e_J\le2$ with $p=q=1$. 
In addition, $p\mapsto \eta(E)^{-1/p}c(\Phi_p,\widetilde{\llbracket \cdot\rrbracket}_{q})$ increases, and $q\mapsto \eta(E)^{1/q}c(\Phi_p,\widetilde{\llbracket \cdot\rrbracket}_{q})$ decreases.  

All of the above properties still hold for $c(\Psi_p,\widetilde{\llbracket \cdot\rrbracket}_{q})$ or  $c(\Psi_p,\llbracket \cdot\rrbracket_{q})$ or  $c(\Phi_p,\llbracket \cdot\rrbracket_{q})$. 
\end{theorem}

\noindent\textbf{Scope of Application}: 
Theorems \ref{thm:critical-inequality} and \ref{thm:critical-inequality2} establish monotonicity inequalities that apply not only to many quantities (such as conductance and torsional rigidity) but also to various types of ambient spaces, including Riemannian manifolds and graph limits. 
These monotonicities naturally give rise to many useful inequalities (see Section \ref{sec:Applications}); therefore, we can regard it as an automatic generator of many important inequalities.

We refer to Section \ref{sec:Applications} 
for some applications of Theorems \ref{thm:critical-inequality} and \ref{thm:critical-inequality2}, and to Section \ref{sec:proof} for their proofs.

\section{Applications}
\label{sec:Applications}
In this section, we combine Theorems \ref{thm:tilde-fg-equal}, \ref{thm:critical-inequality} and \ref{thm:critical-inequality2} to several examples to get useful inequalities on graph limits, hypergraphs, Riemannian manifolds, and metric measure spaces.
\subsection{$L^p$-Cheeger inequalities on graph limits}\label{subsec:Cheeger-monotonicity}

\begin{example}[Cheeger inequality on graphon] \label{exam:nonlocal-graphon}
Given a graphon $W:J\times J\to [0,1]$, 
let $$\varphi_{xy}(A)=\begin{cases}
1,&\text{ if } \#\big(\{x,y\}\cap A\big)=1\\
0,&\text{ otherwise}
\end{cases}
$$ and  
$\psi_{xy}(A,A')=
\varphi_{xy}(A)+\varphi_{xy}(A')
$. 
Then $\widehat{\varphi}_{xy}(f)=\widehat{\psi}_{xy}(f)=
|f(x)-f(y)|$. 
Let $E=J\times J$ be equipped with the measure $\eta$ defined by   $\eta(S):=\int_SW(x,y)dxdy$. 

In this case, $|f|_{\varphi_{xy}}=\|f\|_{\psi_{xy}}=\max\{|f(x)|,|f(y)|\}$, but since $J$ and $E$ are degree-related,  we can actually replace the term $|f|_{\varphi_{xy}}$ by a better quantity $(\frac{|f(x)|^p+|f(y)|^p}{2})^{\frac1p}$ called $p$-mean (this is equivalent to directly using $\widetilde{\llbracket \cdot\rrbracket}_p$ instead of $\llbracket \cdot\rrbracket_p$). 
Then we have
$$\Phi_p(f)=\int_{J\times J}W(x,y)|f(x)-f(y)|^pdxdy\text{ and }\widetilde{\llbracket f\rrbracket}_p=\int_J \deg(x)|f(x)|^pdx. $$
The $\mathcal{Y}_2$-type min-max value $c_2(\Phi_p,\widetilde{\llbracket \cdot\rrbracket}_p)$ unifies many quantities, e.g.   the second smallest eigenvalue $\lambda_2(L_W)=\big(c_2(\Phi_2,\widetilde{\llbracket \cdot\rrbracket}_2)\big)^2$ (where $L_W$ is the graphon Laplacian), and the Cheeger constant $h_W=c_2(\Phi_1,\widetilde{\llbracket \cdot\rrbracket}_1)$, where the latter equality follows from Theorem \ref{thm:tilde-fg-equal}, Example \ref{example:conductance-graphon}, and the kernel reduction lemma in \cite{zhang/dual}.  Therefore, in the setting of Example \ref{exam:nonlocal-graphon}, the min-max extension constants can be precisely rewritten as 
\begin{equation}\label{eq:Cheeger-c2pp}
c_2(\Phi_p,\widetilde{\llbracket \cdot\rrbracket}_p)=\inf_{\text{nonconstant }f\in \Bd}\Big(\frac{\int_J\int_JW(x,y)|f(x)-f(y)|^pdxdy}{\min\limits_{t\in\R}\int_{J} \deg(x)|f(x)-t|^pdx}\Big)^{\frac1p}.
\end{equation}

We can simply apply Theorem \ref{thm:critical-inequality2} to derive the following inequality regarding \eqref{eq:Cheeger-c2pp}. 
\begin{theorem}\label{thm:Cheeger-graphon}
For any $p\in[1,\infty]$ and $r\in[1,\infty)$, we have \begin{equation}\label{eq:p-Cheeger}
2^{1-r}c_2(\Phi_{pr},\widetilde{\llbracket \cdot\rrbracket}_{pr})^r\le c_2(\Phi_p,\widetilde{\llbracket \cdot\rrbracket}_p)\le r\cdot c_2(\Phi_{pr},\widetilde{\llbracket \cdot\rrbracket}_{pr})    
\end{equation}    
\end{theorem}

\end{example}

By taking different values for $p$ and $r$ in Theorem \ref{thm:Cheeger-graphon}, we obtain different inequalities, for example, taking $p=1$ and $r=2$, we have the following:
\begin{corollary}\label{cor:WJ}
Let $W:J\times J\to [0,\infty)$ be a $L^2$-graphon, $h_W$ be its Cheeger constant, and $\lambda_2(L_W)$ be the second smallest eigenvalue of the normalized Laplacian on $W$. Then
$$ \frac{h^2_W}{4} \le \lambda_2(L_W) \le 2h_W\,.$$
\end{corollary}


The above theorem and corollary cover 
the main results in \cite{KM24}. In fact, the first one of the two main results in \cite{KM24} gives the Cheeger inequality on the bounded graphon $W:[0,1]^2\to[0,1]$, which is a special case of Corollary \ref{cor:WJ}.

\begin{example}[Cheeger inequality on graphing]
Given a graphing $G=(J,E)$, 
then, similar to the graphon case (Example \ref{exam:nonlocal-graphon}), it is easy to check that for any $f\in\Bd$, $\widehat{\varphi}_{xy}(f)=|f(x)-f(y)|$. Then we have
$$\Phi_p(f)=\int_E|f(x)-f(y)|^pd\eta(x,y)\text{ and }\widetilde{\llbracket f\rrbracket}_p=\int_J \deg(x)|f(x)|^pd\mu(x). $$
and thus in the graphing setting,
\begin{equation}\label{eq:Cheeger-c2pp-ing}
c_2(\Phi_p,\widetilde{\llbracket \cdot\rrbracket}_p)=\inf_{\text{nonconstant }f\in \Bd}\Big(\frac{\int_E|f(x)-f(y)|^pd\eta(x,y)}{\min\limits_{t\in\R}\int_{J} \deg(x)|f(x)-t|^pdx}\Big)^{\frac1p}
\end{equation}
Again, we  apply Theorem \ref{thm:critical-inequality2} to derive the following
\begin{theorem}\label{thm:Cheeger-graphing}
For any $1\le p\le q <\infty$, 
\begin{equation}\label{eq:p-Cheeger-ing}
\big(\frac{p\, c_2(\Phi_{p},\widetilde{\llbracket \cdot\rrbracket}_{p})}{q}\big)^q\le c_2^q(\Phi_q,\widetilde{\llbracket \cdot\rrbracket}_q)\le  2^{q-p}c_2^p(\Phi_{p},\widetilde{\llbracket \cdot\rrbracket}_{p})    
\end{equation}    
\end{theorem}
In particular, taking $p=1$ and $q=2$ in Theorem \ref{thm:Cheeger-graphing}, we have:
\begin{corollary}
Let $G$ be a graphing, $h_G$ be its Cheeger constant, and $\lambda_G$ be the second smallest eigenvalue of the normalized Laplacian on $G$. Then
$$ \frac{h^2_G}{4} \le \lambda_G \le 2h_G\,.$$
\end{corollary}
The above corollary covers 
the second main result in \cite{KM24}.
\end{example}

\subsection{Spectral bounds for maxcut and bipartiteness ratio on graphons, and conductance on signed graphons}\label{subsec:spectral-bound}

\begin{example}[Maxcut on graphon]

In this example, we work on a standard graphon $W:J\times J\to [0,1]$ with $J=[0,1]$, and we follow the setting of Example \ref{exam:nonlocal-graphon}. In this case, for any admissible set-family $\mathcal{Y}\subset2^\Bd$, \begin{equation}\label{eq:sobolev}
c(\Phi_p,\widetilde{\llbracket \cdot\rrbracket}_q):=\inf\limits_{Y\in \mathcal{Y}}\sup\limits_{f\in Y} \frac{\Big(\int_{J\times J}W(x,y)|f(x)-f(y)|^pdxdy\Big)^{\frac1p} }{\big(\int_J \deg(x)|f(x)|^qdx\big)^{\frac1q}}  .  
\end{equation}
When taking $\mathcal{Y}=\mathcal{Y}_k$ or $\mathcal{Y}=\mathcal{Y}_{\sup}:=\{\Bd\setminus\{0\}\}$ in \eqref{eq:sobolev}, these specific min-max extension constants are exactly the $(p,q)$-Sobolev constants introduced in \cite{Zhang/graphon}. 

By Theorem \ref{thm:tilde-fg-equal} (precisely, its dual version Theorem \ref{th:dual-main}), we have 
\begin{equation}\label{eq:maxcut=p/infty}
c_{\sup}(\Phi_p,\widetilde{\llbracket \cdot\rrbracket}_\infty):=\sup\limits_{f\in \Bd} \frac{\Big(\int_{J\times J}W(x,y)|f(x)-f(y)|^pdxdy\Big)^{\frac1p} }{\|f\|_\infty}   = 2 \big(\mathrm{MaxCut}(W) \big)^{\frac1p}
\end{equation}
which is a generalization of Example \ref{exam:maxcut} (or Corollary \ref{cor:6-equality}) where the maxcut of $W$ is proven to be 
$c_{\sup}(\Phi_1,\widetilde{\llbracket \cdot\rrbracket}_\infty)/2$ with the admissible $\mathcal{Y}$ taken as the single set $\{\Bd\setminus\{0\}\}$. 
Then, Theorem \ref{thm:critical-inequality2} implies 
$c_{\sup}(\Phi_2,\widetilde{\llbracket \cdot\rrbracket}_\infty)\le \|W\|_1^{\frac12}c_{\sup}(\Phi_2,\widetilde{\llbracket \cdot\rrbracket}_2)$. 
Note that with such $\mathcal{Y}$, $c_{\sup}(\Phi_2,\widetilde{\llbracket \cdot\rrbracket}_2)$ indicates the square root of the largest eigenvalue of the  Laplacian on $W$. 
Thus, we have:
\begin{corollary}\label{cor:maxcut-Lapla}
Let $\lambda_{\max}(L_W)$ denote the largest eigenvalue of the  Laplacian on $W$. Then, 
$$4\, \mathrm{MaxCut}(W) \le \|W\|_1\,\lambda_{\max}(L_W).$$
\end{corollary}
\end{example}

\begin{example}[Bipartiteness ratio on graphon]
By Example \ref{exam:bipartiteness}, the bipartiteness ratio $\beta_W$ on $W$ equals $$\inf_{f\in \Bd}
\frac{\int_J\int_JW(x,y)|f(x)+f(y)|dxdy}{2\int_J\deg(x)|f(x)|dx}.$$ 
To obtain the expression of $\Psi_p$, we note that $\widehat{\psi}_{xy}(f)=|f(x)+f(y)|$, where 
$$\psi_{xy}(A,A')=\begin{cases}
0,&\text{ if } x\in A,y\in A'\text{ or }x\in A',y\in A\text{ or }\{x,y\}\cap (A\cup A')=\emptyset,\\
1,&\text{ if }\#\big(\{x,y\}\cap (A\cup A')\big)=1,\\
2,&\text{ otherwise}.
\end{cases}
$$
We can then take $\mathcal{Y}_1=\big\{\{-f,f\}:f\in \Bd\text{ with }\|f\|>0\big\}$ to obtain 
$$c_1(\Psi_p,\widetilde{\llbracket \cdot\rrbracket}_q)=\inf\limits_{Y\in \mathcal{Y}_1}\sup\limits_{f\in Y} \frac{\Big(\int_{J\times J}W(x,y)|f(x)+f(y)|^pdxdy\Big)^{\frac1p} }{\big(\int_J \deg(x)|f(x)|^qdx\big)^{\frac1q}}.$$
With such $\mathcal{Y}_1$,  
the bipartiteness ratio $\beta_W$ equals $c_1(\Psi_1,\widetilde{\llbracket \cdot\rrbracket}_1)/2$, and the square root of the smallest eigenvalue of the signless Laplacian on $W$ is equal to $c_1(\Psi_2,\widetilde{\llbracket \cdot\rrbracket}_2)$. 

Theorem \ref{thm:critical-inequality2} implies the following inequality: 
\begin{theorem}\label{thm:Cheeger-graphon-dua}
For any $p,r\ge1$, we have \begin{equation}\label{eq:p-Cheeger-dual}
2^{1-r}c_1(\Psi_{pr},\widetilde{\llbracket \cdot\rrbracket}_{pr})^r\le c_1(\Psi_p,\widetilde{\llbracket \cdot\rrbracket}_p)\le r\cdot c_1(\Psi_{pr},\widetilde{\llbracket \cdot\rrbracket}_{pr})    
\end{equation}    
\end{theorem}

Taking $p=1$ and $r=2$ in the above theorem, we obtain the spectral gap inequality relating the bipartiteness ratio and the smallest eigenvalue of the normalized signless Laplacian on $W$. This can be written as the following corollary. 
\begin{corollary}\label{cor:bipart-maxcut}
Let $W:[0,1]^2\to[0,1]$ be a graphon, $\beta_W$ be its bipartiteness ratio, and $\lambda_{\max}(L_W)$ be the largest eigenvalue of the normalized Laplacian on $W$. Then
$$ \beta_W^2\le 4-\lambda_{\max}(L_W) \le 4\beta_W$$
\end{corollary}
Together with Example \ref{exam:bipartiteness}, we have $4\,\text{MaxCut}(W)\le 4\|W\|_1(1-\beta_W)\le \|W\|_1\lambda_{\max}(L_W)$ which refines Corollary \ref{cor:maxcut-Lapla}.
\end{example}

\begin{example}[Conductance on signed graphons]
    
Finally, we investigate the conductance on signed graphons \cite{Zhang/graphon}. 
For a signed graphon $W:J\times J\to [-1,1]$, the signed conductance of $W$ is defined as
$$h(W):=\inf_{\text{measurable }A,B\subset J,A\cap B=\varnothing} 
\psi(A,B)
$$
where 
$$\psi(A,B)=
\frac{4\eta(W_+[A\times B])+2\eta \big(W_-[(A\times A)\cup(B\times B)]\big)+2\eta((A\cup B)\times (A\cup B)^c)}{\mu(A\cup B)}\,,$$
see Example \ref{example:frustration} for the meaning of notions $W_\pm[\cdot]$ and others. 
Taking 
$$\psi_{xy}(A,A')=
\big|1_A(x)-1_{A'}(x)-\mathrm{sgn}(W(x,y))(1_A(y)-1_{A'}(y))\big| \,,
$$
we have 
$\widehat{\psi}_{xy}(f)=\big|f(x)-\mathrm{sgn}(W(x,y))f(y)\big|$. By Theorem \ref{thm:critical-inequality2}, the quantity 
$$c(\Psi_p,\widetilde{\llbracket \cdot\rrbracket}_p)=\inf\limits_{f\in \Bd} \frac{\Big(\int_{J\times J}|W(x,y)|\big|f(x)-\mathrm{sgn}(W(x,y))f(y)\big|^pdxdy\Big)^{\frac1p} }{\big(\int_J \deg(x)|f(x)|^pdx\big)^{\frac1p}}$$
satisfies certain monotonicity inequalities.
\begin{prop} The min-max quantity $c(\Psi_p,\widetilde{\llbracket \cdot\rrbracket}_p)$ satisfies the following:

\begin{itemize}
\item $c(\Psi_1,\widetilde{\llbracket \cdot\rrbracket}_1)=h(W)$
\item $c(\Psi_2,\widetilde{\llbracket \cdot\rrbracket}_2)=\sqrt{\lambda_1(L_W)}$ 
\item $2^{-p}c(\Psi_p,\widetilde{\llbracket \cdot\rrbracket}_p)^p$ is decreasing with respect to $p\in[1,\infty)$
\item $p\cdot c(\Psi_p,\widetilde{\llbracket \cdot\rrbracket}_p)$ is increasing with respect to $p\in[1,\infty)$
\end{itemize} 
In particular, for a signed graphon $W$, the following Cheeger-like inequality holds:
$$ \frac{h^2(W)}{4} \le \lambda_1(L_W) \le 2h(W)$$
where $\lambda_1(L_W)$ is the smallest eigenvalue of the signed Laplacian on $W$, and $h(W)$ is the signed conductance of $W$.
\end{prop}

\end{example}

\subsection{
Hypergraph $p$-Laplacian and hyperedge expansion
}

In recent years, there has been growing interest in applying hypergraph tools to the analysis of real-world data. To this end, various methods have been proposed aimed at extending spectral graph theory to the field of hypergraphs and their analysis. 
It should be noted that there are multiple ways to generalize graph operators to the hypergraph setting. 
For example, there are two different $p$-Laplacian for hypergraphs that are widely used; 
one is closely related to the incidence matrix \cite{Stokke25,Fazeny24,JZ22}, whilst the other is closely associated with the vertex expansion or hyperedge expansion 
\cite{Chan18,Louis12}. 
In this section, we shall demonstrate that the latter can be naturally incorporated into the Choquet-type extension framework, 
thereby yielding richer results. 

Given an oriented hypergraph $(V,E)$ (see \cite{Jost/Mulas/Zhang26} for the definition) satisfying  $e_{in}\ne\emptyset\ne e_{out}$ and $\#(e_{in}\cup e_{out})\ge 2$ for all $e\in E$, let $\varphi_{e}:\mathcal{P}(V)\to \R$ be defined by
\begin{equation*}
\varphi_{e}(A)=\begin{cases}
1,&\text{ if } e_{in}\cap A\ne\emptyset\ne e_{out}\setminus A\text{ \; or\; }e_{out}\subset  A\subset V\setminus e_{in},\\
0,&\text{ otherwise},
\end{cases}
\end{equation*}
where $e_{in}:=\{\text{inputs of }e\}$ and  $e_{out}:=\{\text{outputs of }e\}$, and $\mathcal{P}(V)$ is the power set of $V$. 
One can check that $\varphi_e(A)=0$ if and only if $A\supset e_{in}\cup e_{out}$ or $A\cap(e_{in}\cup e_{out})=\varnothing$.
Then the Choquet extension of $\varphi_{e}$ is determined by 
$\widehat{\varphi}_{e}(f)=|\max\limits_{u\in e_{in}}f(u)-\min\limits_{v\in e_{out}}f(v)|$, and we have the min-max extension constant
$$ c_k(\Phi_p,\widetilde{\llbracket \cdot\rrbracket}_q):=\inf\limits_{Y\in \mathcal{Y}_k}\sup\limits_{f\in Y} \frac{\big(\sum_{e\in E}|\max\limits_{u\in e_{in}}f(u)-\min\limits_{v\in e_{out}}f(v)|^p\big)^{\frac1p}}{\big(\sum_{v\in V}\deg(v)|f(v)|^q\big)^{\frac1q}}.$$

The hyperedge expansion of $(V,E)$ is defined as 
\begin{equation}\label{eq:expansion-Ch}
h:=\min\limits_{A\in\mathcal{P}(V)\setminus\{\emptyset,V\}}\frac{\#(\partial A)}{\min\{\vol(A),\vol(V\setminus A)\}}\end{equation}
see \cite[Chapter 6]{Jost/Mulas/Zhang26}, where we adopt  the volume  $\vol(A):=\sum_{e\in E}\#(e\cap A)=\sum_{i\in A}\deg(i)$, the degree $\deg(i):=\# \{e\in E: i\in e\}$, and the boundary set \begin{equation*}\partial A:=\{e\in E: e_{in}\cap A\ne\emptyset\ne e_{out}\setminus A\text{ \; or\; }e_{out}\subset  A\subset V\setminus e_{in}\}.\end{equation*}


When taking $\mathcal{Y}=\mathcal{Y}_2$ or $\mathcal{Y}=\{Y_g:g\ne0\}$ where $Y_g:=\{g+c:c\in\R\}$, and $p=q=1$, we have  $c(\Phi_1,\widetilde{\llbracket \cdot\rrbracket}_1)=h$, and by adopting the hypergraph $p$-Laplacian introduced in \cite[Chapter 6]{Jost/Mulas/Zhang26}, 
the second eigenvalue $\lambda_2(\Delta_p)$ satisfies   $\big(c(\Phi_p,\widetilde{\llbracket \cdot\rrbracket}_p)\big)^p=\lambda_2(\Delta_p)$. 
In a similar manner, $\big(c_k(\Phi_p,\widetilde{\llbracket \cdot\rrbracket}_p)\big)^p=\lambda_k(\Delta_p)$ for any $k$. 
Theorem \ref{thm:critical-inequality2} implies the following inequality. 
\begin{theorem}
\label{thm:p-Lap-Cheeger}
The min-max eigenvalues of the $p$-Laplacians on oriented hypergraphs satisfy the following inequality
$$ \frac{2^{p-1}}{p^p} \big(\lambda_k(\Delta_1)\big)^p\le \lambda_k(\Delta_p)\le2^{p-1}\lambda_k(\Delta_1) $$
and, in particular, we have the Cheeger inequality
\begin{equation}\label{eq:Cheeger-chemical-hyper-p}
2^{p-1}\frac{h^p}{p^p}\le \lambda_2(\Delta_p)\le 2^{p-1}h
\end{equation}
where $h$ is the hyperedge expansion defined in \eqref{eq:expansion-Ch}, and $\lambda_k(\Delta_p)$ denotes the $k$-th min-max eigenvalue of the $p$-Laplacian on oriented hypergraphs.
\end{theorem}

\begin{remark}\label{remark:p-Lap-} 
Theorem \ref{thm:p-Lap-Cheeger} includes the following special cases: 
\begin{itemize}
\item Taking $p=2$ and letting  $e_{in}=e_{out}$ for any $e\in E$, we get Louis's hypergraph Laplacian  \cite{Louis15} and the Cheeger inequality therein. 
\item Taking $(V,E)$ as a graph (i.e., $e_{in}=e_{out}$ and $\#e_{in}=2$),  Theorem \ref{thm:p-Lap-Cheeger} implies the Cheeger inequality for the graph $p$-Laplacian \cite{Amghibech03}. 
\item Taking $p=1$ and letting  $e_{in}=e_{out}$ for any $e\in E$, we get the total variation on hypergraphs  \cite{TVhyper-13}.
\end{itemize}
\end{remark}


\subsection{Torsional rigidity and eigenvalues of $p$-Laplacian on Riemannian manifolds}\label{subsec:pLap-monotonicity}

\begin{example}
Let $r> 0$. Let $E=J=M$ be a Riemannian manifold of dimension $n$ and $\B$ be the collection of Borel sets. For any $x\in M$ and $A\in\B$, define $$\varphi_x(A)=\frac{\mathcal{H}^{n-1}(\partial A\cap B_r(x))}{\mathcal{H}^n(B_r(x))}.$$ 
Then 
$\widehat{\varphi}_x(f)=\frac{1}{\mathcal{H}^n(B_r(x))}\int_{B_r(x)}|\nabla f(y)|dy$ and $|f|_{\varphi_x}=\max\limits_{y\in B_r(x)}|f(y)|$. Thus, it follows from Theorem \ref{thm:critical-inequality} that 
$p\cdot c(\Phi_{p,r},\llbracket \cdot\rrbracket_{p,r})$ is increasing with respect to $p\ge1$, where $$\Phi_{p,r}(f):=\left( \int_M\left(\int_{B_r(x)}\frac{|\nabla f(y)|}{\mathcal{H}^n(B_r(x))}dy\right)^pdx
\right)^{\frac1p}\;\text{ and }\; \llbracket f\rrbracket_{p,r}=\left( \int_M\max\limits_{y\in B_r(x)}|f(y)|^pdx
\right)^{\frac1p}.$$
Taking $r\to0^+$, we have
$$\lim\limits_{r\to0^+}\Phi_{p,r}(f)=\|\nabla f\|_p\;\text{ and }\lim\limits_{r\to0^+}\llbracket f\rrbracket_{p,r}=\|  f\|_p.
$$

We shall adopt $\mathcal{Y}=\mathcal{Y}_k$ defined in \eqref{eq:Y_k}. 
Then, using standard approaches in variational analysis, it is not difficult to show that for any $k\in \mathbb{N}_+$, the $\mathcal{Y}_k$-type min-max extension constants 
satisfy
$$
\lim\limits_{r\to0^+}c_k(\Phi_{p,r},\llbracket \cdot\rrbracket_{p,r})=c_k( \|\nabla \cdot\|_p,\| \cdot\|_p)$$
where $\big(c_k( \|\nabla \cdot\|_p,\| \cdot\|_p)\big)^p$ equals the $k$-th min-max variational eigenvalue of $p$-Laplacian on $M$.
Therefore, we obtain:
\begin{theorem}
The quantity $p\cdot c_k( \|\nabla \cdot\|_p,\| \cdot\|_p)$ is increasing with respect to  $p\in[1,\infty)$, where $c_k( \|\nabla \cdot\|_p,\| \cdot\|_p)$ equals the $p$-root of  the $k$-th min-max variational eigenvalue of $p$-Laplacian on the Riemannian manifold $M$.  
\end{theorem}
The above result covers Theorem 2 in \cite{Jose21} and Theorem 3.2 in \cite{Lindqvist93}. 

For a compact Riemannian manifold without boundary, $c_2( \|\nabla \cdot\|_1,\| \cdot\|_1)$ is the Cheeger constant; while for a compact Riemannian manifold with boundary, $c_1( \|\nabla \cdot\|_1,\| \cdot\|_1)$ is the Cheeger constant \cite{Cheeger70,Yau75}. In any case, $c_k( \|\nabla \cdot\|_2,\| \cdot\|_2)=\sqrt{\lambda_k(L_M)}$, where $\lambda_k(L_M)$ is the $k$-th eigenvalue of the Laplace-Beltrami operator on $M$. 
In consequence, the above result covers the classical Cheeger inequality \cite{Cheeger70}.
\end{example}

Theorem \ref{thm:critical-inequality} implies the following proposition.

\begin{prop}
For any $k\ge 1$, $r>0$, for any $p,q,s,t\ge1$ with $ps'=qt'$, 
where $s'$ and $t'$ are the H\"older conjugates of $s$ and $t$, respectively,  we have the following inequality 

$$c_k(\Phi_{p,r},\llbracket \cdot\rrbracket_{q,r})\le t\cdot c_k(\Phi_{ps,r},\llbracket \cdot\rrbracket_{qt,r})\,.$$

\end{prop}
Taking $r\to0^+$ in the above proposition, we have the following:
\begin{theorem}\label{thm:grad-pqst}
For any $k\ge 1$, 
  $$c_k( \|\nabla \cdot\|_p,\| \cdot\|_q)\le t\cdot c_k( \|\nabla \cdot\|_{ps},\| \cdot\|_{qt})$$
whenever $p,q,s,t\ge1$ and $ps'=qt'$,  where $s'$ and $t'$ are the H\"older conjugates of $s$ and $t$, respectively. 
\end{theorem}

Theorem \ref{thm:critical-inequality} also implies: 
\begin{prop}\label{pro:volM-mono}For a Riemannian manifold $M$ of finite volume, the quantity $\vol(M)^{1/q-1/p}c_k( \|\nabla \cdot\|_p,\| \cdot\|_q)$ increases in $p$, and decreases in $q$, where $\vol(M)$ denotes the volume of $M$. 
\end{prop}

\begin{example}[Torsional rigidity
]
Given an open bounded subset $\Omega$ of $\mathbb{R}^N$, the first Dirichlet eigenvalue $\lambda(\Omega)$ and the torsional rigidity $T(\Omega)$ are defined as follows:
\[\lambda(\Omega):=\min_{u\in H_{0}^{1}(\Omega)\backslash \{0\}}\frac{\int_{\Omega}|\nabla u(x)|^{2}\,\mathrm{d}x}{\int_{\Omega}|u(x)|^{2}\,\mathrm{d}x}
\quad\text{and}\quad
T(\Omega):= \max_{u\in H_{0}^{1}(\Omega)\backslash \{0\}}\frac{\left(\int_{\Omega}u(x)\,\mathrm{d}x\right)^{2}}{\int_{\Omega}|\nabla u(x)|^{2}\,\mathrm{d}x}.\]
The classical torsional rigidity has been investigated by many authors (see \cite{Berg21,Briani22,Crasta,Della}). 
Our result finds some new properties on $p$-torsional rigidity in  Riemannian manifolds with boundary. 
Let \[\lambda_p(M):=\min_{u\in W_{0}^{1,p}(M)\backslash \{0\}}\frac{\int_{\Omega}|\nabla u(x)|^{p}\,\mathrm{d}x}{\int_{\Omega}|u(x)|^{p}\,\mathrm{d}x}
\quad\text{and}\quad
T_p(M):= \max_{u\in W_{0}^{1,p}(M)\backslash \{0\}}\frac{\left|\int_{\Omega}u(x)\,\mathrm{d}x\right|^{p}}{\int_{\Omega}|\nabla u(x)|^{p}\,\mathrm{d}x}.\]
\end{example}

In particular, taking $M=\Omega\subset\R^n$ as a bounded Euclidean domain or taking $M$ as a compact Riemannian manifold with boundary, the $p$-torsional rigidity can be formulated as $T_p(M):=\big(c_1( \|\nabla \cdot\|_p,\| \cdot\|_1)\big)^{-p}$, and the first $p$-Laplacian eigenvalue can be written as $\lambda_p(M):=\big(c_1( \|\nabla \cdot\|_p,\| \cdot\|_p)\big)^{p}$. 
Then, the monotonicity property in Proposition \ref{pro:volM-mono} implies: 
\begin{corollary}Let $M$ be a compact Riemannian manifold with boundary. Then, for any $ p\ge 1$, we have 
$$T_p(M)\lambda_p(M)\le \vol(M)^{p-1}\,.$$    
\end{corollary}
This includes the classical P\'olya-Szeg\"o inequality (when $p=2$) on torsional rigidity \cite{Polya51}, and the $p$-P\'olya-Szeg\"o inequality \cite{Della}, as special cases. 







Moreover, Theorem \ref{thm:grad-pqst} yields: 
\begin{corollary}For any $p\in(1,\infty]$, 
$$c_1( \|\nabla \cdot\|_p,\| \cdot\|_1)\le p'\cdot c_1( \|\nabla \cdot\|_{\infty},\| \cdot\|_{p'})$$    
where $p'$ is the H\"older conjugate of $p$. 
\end{corollary}
This corollary, applying to the setting of Euclidean domains, can be reformulated as
\begin{prop}\label{prop:torsional}
Given a bounded Euclidean domain $\Omega$, for any $p> 1$, we have 
$$ T_p(\Omega)\ge (1-\frac1p)^p \big(\int_\Omega \mathrm{dist}_\Omega^{p'}(x)dx\big)^{p-1} $$
where $\mathrm{dist}_\Omega$ denotes the distance function to $\partial\Omega$, i.e., $\mathrm{dist}_\Omega(x):=\inf_{y\in\partial\Omega}|x-y|$ with $|\cdot|$ being the standard Euclidean norm. 
In particular, 
$$  T(\Omega)\ge\frac14  \int_\Omega \mathrm{dist}_\Omega^2(x)dx .$$
\end{prop}

We omit the detailed proofs of all the results in this section, since we have given the key ingredients‌ for the proofs by explanations and discussions.

\subsection{A monotonicity inequality on metric measure space}

Following the terminology and notation in \cite{DePonti21b}, in this section we consider a complete and separable metric space $(X, \mathsf{d})$ endowed with a finite Borel measure $\mm$ (i.e., $\mm(X)<\infty$). The triple $(X, \mathsf{d}, \mm)$ is called \emph{metric measure space}, m.m.s. for short.
The space of real-valued Lipschitz (resp.  
Lipschitz with bounded support, Lipschitz on bounded sets) functions over $X$ will be denoted by $\mathsf{Lip}(X)$ (resp. 
$\mathsf{Lip}_{bs}(X)$, $\mathsf{Lip}_{loc}(X)$). The local Lipschitz constant (or slope) of a function $f:X\rightarrow \R$ at $x\in X$ is defined by
\begin{equation*}
\mathrm{lip}(f)(x):=\limsup_{y\rightarrow x} \frac{|f(y)-f(x)|}{\di(y,x)}\, ,
\end{equation*}  
with the convention $\mathrm{lip}(f)(x)=0$ if $x$ is an isolated point.

We introduce the following relevant definitions:  for any $1<p<\infty$ we set
\begin{equation}\label{eq:defla1Intro}
\lambda_{2,p}(X)= \inf\bigg\{\frac{\int_X \mathrm{lip}(f)^p \di\mm}{\int_X |f|^p\,\di\mm}: \ 0\not\equiv f\in \mathsf{Lip}_{bs}(X), \int_X |f|^{p-2}f\, \di\mm=0\bigg\}
\end{equation} 
where $\lambda_{2,p}(X)$ is called the $p$-spectral gap \cite{Cavalletti}. 
This can be generalized to the case where $k\ge 1$ and $1\le p<\infty$:
\begin{equation}\label{eq:lambda-k,p}
\lambda_{k,p}(X)= \inf\bigg\{\sup_{f\in S}\frac{\int_X \mathrm{lip}(f)^p \di\mm}{\int_X |f|^p\,\di\mm}:  \text{centrally symmetric }S\subset \mathsf{Lip}_{bs}(X)\setminus\{0\}, \mathrm{genus}(S)\ge k\bigg\}.    
\end{equation}
Note that $\lambda_{2,p}(X)$ has many equivalent representations, apart from \eqref{eq:lambda-k,p} for $k=2$, and \eqref{eq:defla1Intro}, we also have
\begin{equation*}
\lambda_{2,p}(X)=\inf\left\{\frac{1}{c_p^p(f)}\int_X \mathrm{lip}(f)^p\di\mm \ : \ f\in \mathsf{Lip}_{bs}, \ f \ \textrm{non $\mm$-a.e. constant} \right\},
\end{equation*}
where 
$$c^p_p(f):=\inf_{a\in \mathbb{R}}\int_X |f-a|^p\,\di\mm\, .$$

\medskip

A function $f\in L^1(X,\mm)$ belongs to the 
space ${\rm BV}(X,\di,\mm)$ of functions with bounded variation if there exists a sequence $(f_n)_{n\in\mathbb{N}}\in \mathsf{Lip}_{loc}(X)$ converging to $f$ in $L^1(X,\mm)$ and such that
$$\limsup_{n\to \infty}\int_X \mathrm{lip}(f_n)\, \di\mm <+\infty\, .$$

If $A\subset X$ is open, and $f\in {\rm BV}(X,\di,\mm)$, we define 
\begin{equation*}\label{eq:defTV}
|Df|(A):=\inf\bigg\{\liminf_{n\rightarrow \infty}\int_A \mathrm{lip}(f_n)\,\di\mm: f_n\in \mathsf{Lip}_{loc}(X), f_n\rightarrow f \ \mathrm{in} \ L^1(A,\mm)\bigg\}.
\end{equation*}

It is known that this function is the restriction to open sets of a finite Borel measure, called \textit{total variation} of $f$ and denoted by $|Df|$. By the  definition of $|Df|(X)$, it is immediate to see that for all $f, f_n\in {\rm BV}(X,\di,\mm)$,
\begin{equation*}\label{eq:lscbv}
|Df|(X)\leq \liminf_n |Df_n|(X) \quad \textrm{whenever} \ f_n\rightarrow f \ \textrm{in} \ L^1(X,\mm).
\end{equation*}
Moreover, for all $1$-Lipschitz $\varphi:\R\to \R$  with $\varphi(0)=0$ we have 
$$|D(\varphi\circ f)|(X)\leq |Df|(X).$$

Given a Borel set $A$ of finite measure, we say that $A$ is a set of finite perimeter if $\chi_A\in {\rm BV}(X,\di,\mm)$ and we set the \textit{perimeter}  $\Per(A):=|D\chi_A|(X)$, i.e.
\begin{equation*}\label{def:per}
\mathrm{Per}(A)=\inf\bigg\{\liminf_{n\rightarrow \infty}\int_X \mathrm{lip}(f_n)\,\di\mm: f_n\in \mathsf{Lip}_{loc}(X), f_n\rightarrow \chi_A \ \mathrm{in} \ L^1(X,\mm)\bigg\}
\end{equation*}
where $\chi_A:X\to\{0,1\}$ is the indicator function of the set $A\subset X$.

The \textit{Cheeger constant} of the metric measure space $(X,\mathsf{d},\mm)$ is defined as 
\begin{equation*}\label{eq:defChConst}h(X):=\inf  \left\{\frac{\Per(A)}{\mm(A)}\, :\, A\subset X \text{ Borel with $0<\mm(A)\leq \mm(X)/2$} \right\}\,.
\end{equation*}
One can use Theorem \ref{thm:tilde-fg-equal} to easily derive $\lambda_{2,1}(X)=h(X) $. 
\begin{lemma}[Coarea inequality and coarea formula, Proposition 2.6 in \cite{DePonti21b}]\label{prop:lemma}
Let $(X,\di)$ be a complete metric space and let $\mm$ be a non-negative Borel measure finite on bounded subsets. Let $f\in {\mathsf {Lip}}_{bs}(X)$, $f:X\to [0,\infty)$ and set $M=\sup_{X} f$. 
Then for $\mathcal{L}^{1}$-a.e. $t>0$ the set $\{f> t \}$ has finite perimeter and 
\begin{equation*}
\int_{0}^{M} \Per(\{f>t\}) \, \di t \leq \int_{X} |{\rm lip}(f)| \, \di \mm.
\end{equation*}

If in addition $(X,\di)$ is  separable, then the coarea formula for $\rm{BV}$ functions holds, i.e. for every $f:X\to [0,\infty)$ with  $f\in {\rm BV}(X, \di,\mm)$ it holds
\begin{equation*}
\int_{0}^{\infty} \Per(\{f>t\}) \, \di t= |Df|(X).
\end{equation*}
\end{lemma}

By the above Coarea inequality in \cite{DePonti21b}, and our Theorem \ref{thm:critical-inequality}, and using a method similar to that in Section \ref{subsec:pLap-monotonicity}, we obtain the following monotonicity inequality.
\begin{theorem}
    \label{thm:mono-mms}
Let $(X,\di,\mm)$ be a complete metric measure space with $\mm(X)<\infty$. Then for any positive integer $k$, 
\begin{equation*}
p\big(\lambda_{k,p}(X)\big)^{\frac{1}{p}}\le q\big(\lambda_{k,q}(X)\big)^{\frac{1}{q}}\,  \qquad\textrm{for every} \ \ 1\le p<q<\infty\, .
\end{equation*}

\end{theorem}

\section{Proofs and Supplementary}\label{sec:proof}

\begin{prop}\label{pro:setpair-generalize-original}
Suppose that $\psi$ is defined as  $\psi(A,A')=\varphi(A)+\varphi(J\setminus A')-\varphi(J)$ for any disjoint measurable subsets $A$ and $A'$, and assume that $\varphi\{f\ge t\}=\varphi\{f> t\}$ for a.e. $t\in\R$, then $\widehat{\psi}(f)=\widehat{\varphi}(f)$. 
\end{prop}

\begin{proof}
The proof is quite direct. Note that
\begin{align*}
\widehat{\psi}(f)&=\int_0^{\|f\|_\infty} \psi(f\ge t,f\le-t)dt 
\\&=\int_0^{\|f\|_\infty}\big(\varphi\{f\ge t\}+\varphi\{f>- t\}-\varphi(J)\big)dt\\
&=\int_{-\|f\|_\infty}^{\|f\|_\infty}\varphi\{f\ge t\}dt- \|f\|_\infty \varphi(J)
\\&=\int_{\inf(f)}^{\sup(f)}\varphi\{f\ge t\}dt +\inf(f)\varphi(J)=\widehat{\varphi}(f) 
\end{align*}
where we used $\|f\|_\infty=\max\{-\inf(f),\sup(f)\}$ and $\varphi(\varnothing)=0$.    
\end{proof}

Below, we give some additional remarks on some results: 
\begin{itemize}
\item 
For Corollary \ref{cor:WJ}, we should point out that since $\mu(J)<\infty$, an $L^2$-graphon must be also an $L^p$-graphon for any $1\le p\le 2$. 
The proof of Theorems \ref{thm:critical-inequality2} and \ref{thm:Cheeger-graphon} for the standard bounded graphons is still available for that of $L^2$-graphons. And thus, Corollary \ref{cor:WJ} follows directly from Theorem \ref{thm:Cheeger-graphon}. 
\item 
For \eqref{eq:c-monotone1} in Theorem \ref{thm:critical-inequality}, we note that $ps'_1=qt'_1$ and $ps_1s_2'=qt_1t_2'$ imply $p(s_1s_2)'=q(t_1t_2)'$, where all the constants $p,q,s_1,s_2,t_1,t_2$ are in the interval $[1,\infty]$, and $s'$ denotes the H\"older conjugate of a given $s\in[1,\infty]$. Thus, $c(\Phi_p,\llbracket \cdot\rrbracket_{q})\le t_1\cdot  c(\Phi_{ps_1},\llbracket \cdot\rrbracket_{qt_1})\le t_1t_2\cdot  c(\Phi_{ps_1s_2},\llbracket \cdot\rrbracket_{qt_1t_2})$, indicating that it is indeed a monotonicity inequality. 
By an alternative formulation, we have $(\min\{q^{-1},p^{-1}\},\infty]\ni z\mapsto (zp)'\cdot c(\Phi_{p(zq)'},\llbracket \cdot\rrbracket_{q(zp)'})$ is nonincreasing.
\end{itemize}

\subsection{Proof of Theorem \ref{thm:bisubmodular}}

\begin{lemma}\label{lem:for-Thm-equ}
Let $\psi:\B_2\to\R$ be a bounded bisubmodular function with $\psi(\varnothing,\varnothing)=0$. Let $(A_1,A_1')$, $\cdots$, $(A_n,A_n')\in \B_2$, $a_1,\cdots,a_n\in \mathbb{N}$. Then
\begin{equation}
\label{eq:psi-extension-<sum}
\widehat{\psi}\left(\sum_{i=1}^n a_i({\bf1}_{A_i}-{\bf1}_{A_i'})\right)\le \sum_{i=1}^n a_i\psi(A_i,A_i').    
\end{equation}
If the setpairs $(A_i,A_i')$ 
form a chain, that is, $A_1\subset A_2\subset\cdots\subset A_n$ and $A_1'\subset A_2'\subset\cdots\subset A_n'$, then the inequality \eqref{eq:psi-extension-<sum} reduces to an equality for any setpair function $\psi$. 
\end{lemma}

\begin{proof}Following the proof of Theorem 3.1 in \cite{Lovasz25}, we use the basic combinatorial technique
called “uncrossing”. 
We may assume that the ground set $J$ is finite, since we may merge the atoms
of the set-algebra generated by $\{(A_i,A_i')\}_{i=1}^n$ to single points. The assertion about
equality is easy to be verified by observing 
\[\int_0^{
\infty}\psi\Big(\sum_{i=1}^n a_i{\bf1}_{A_i}\ge t,-\sum_{i=1}^n a_i{\bf1}_{A_i'}\le- t\Big)dt=\sum\limits_{i=1}^{n-1}\int^{\sum\limits_{j=i}^na_j}_{\sum\limits_{j=i+1}^na_j}\psi(A_i,A_i')dt+\int_0^{a_n}\psi(A_n,A_n')dt.\]
For the
general case, let $\A$ be the multiset consisting of $a_i$ copies of $(A_i,A_i')$, and let $|\A|=\sum_ia_i$. Then $\sum_{i=1}^n a_i({\bf1}_{A_i}-{\bf1}_{A_i'})=\sum_{(A,A')\in \A}({\bf1}_A-{\bf1}_{A'})$ and we want to prove that $$\widehat{\psi}\left(\sum_{(A,A')\in \A}({\bf1}_A-{\bf1}_{A'})\right)\le \sum_{(A,A')\in \A} \psi(A,A').$$
Suppose that we find two setpairs $(A_1,A_1'),(A_2,A_2')\in\A$ such that neither one of them
contains the other. Replace one copy of $(A_1,A_1')$ and of $(A_2,A_2')$ by $(A_1,A_1')\vee (A_2,A_2')$ and $(A_1,A_1')\wedge (A_2,A_2')$, and let $\A'$ be the resulting multiset. Then clearly, $\sum_{(A,A')\in \A'}({\bf1}_A-{\bf1}_{A'})=\sum_{(A,A')\in \A}({\bf1}_A-{\bf1}_{A'})$, and 
$$\sum_{(A,A')\in \A'} \psi(A,A')\le \sum_{(A,A')\in \A} \psi(A,A') $$
by bisubmodularity. Let us repeat this transformation as long as we can. Since we stay with subsets of a finite set and $|\A|$ does not change, but the quantity $\sum_{(A,A')\in\A}\sqrt{|A|+|A'|}$ strictly decreases at each step, the procedure must stop after a finite
number of iterations with a multiset that is a chain. As remarked above, in this case equality holds, which proves the inequality in the lemma. 
\end{proof}

\textbf{Proof of Theorem \ref{thm:bisubmodular}}. 
Since $\widehat{\psi}({\bf1}_{B}-{\bf1}_{B'})=\psi(B,B')$, $\forall (B,B')\in \B_2$, there is no difficulty to show that the convexity of $\widehat{\psi}$ implies the bisubmodularity of $\psi$. 

Note that Lemma \ref{lem:for-Thm-equ}  indeed shows a special case Theorem \ref{thm:bisubmodular}. 

For general $f$ and $g$, 
if they are integer-valued stepfunctions, then we express them by their layer cake representation with all the coefficients $a_i$ and $b_j$ nonnegative, and apply the above Lemma to get the inequality $\widehat{\psi}(f+g)\le \widehat{\psi}(f)+\widehat{\psi}(g)$. For rational-valued stepfunctions, the inequality follows by scaling. The general case follows via approximation by stepfunctions and the Lipschitz continuity of $\widehat{\psi}$.

Finally, by homogeneity and the additivity, we derive the convexity of $\widehat{\psi}$.

\subsection{Proof of Theorem \ref{thm:tilde-fg-equal}}
\label{sec:proof-main-equal}

Case 1: Suppose that $\varphi_1(J)=\varphi_2(J)=0$, $\widetilde{\varphi}_1$ is a super-Choquet extension  of $\varphi_1$, and $\widetilde{\varphi}_2$ is a sub-Choquet extension  of $\varphi_2$.

In this case,  we have $\widetilde{\varphi}_1(1_A)=\varphi_1(A)$ and  $\widetilde{\varphi}_2(1_A)=\varphi_2(A)$ for any $A\in\B$, as well as $\widetilde{\varphi}_2(f)\le \widehat{\varphi}_2(f)$ and $\widetilde{\varphi}_1(f)\ge \widehat{\varphi}_1(f)$ for any $f\in\Bd$. 
This implies 
\begin{align*}
\inf_{f\in \Bd \text{ with }\widetilde{\varphi}_2(f)>0}\frac{\widetilde{\varphi}_1(f)}{\widetilde{\varphi}_2(f)}\le \inf_{{\bf1}_A\in \Bd \text{ with }\widetilde{\varphi}_2({\bf1}_A)>0}\frac{\widetilde{\varphi}_1({\bf1}_A)}{\widetilde{\varphi}_2({\bf1}_A)} =\inf_{A\in\B\text{ with }\varphi_2(A)>0}\frac{\varphi_1(A)}{\varphi_2(A)} .
\end{align*}
For the other direction, given  $f\in \Bd $ with $\widetilde{\varphi}_2(f)>0$, let $S_+(\varphi_2,f):=\{t\ge \inf(f): \varphi_2\{f\ge t\}>0\}$. Clearly, $\widehat{\varphi}_2(f)>0$ and the Lebesgue measure of $S_+(\varphi_2,f)$ is positive. 

Denote by $$C:=\inf\limits_{t\in S_+(\varphi_2,f)} \frac{\varphi_1\{f\ge t\}}{\varphi_2\{f\ge t\}}.$$

Then, \begin{equation}\label{eq:t0}\varphi_1\{f\ge t\}\ge C\varphi_2\{f\ge t\}   \end{equation} for any $t\in S_+(\varphi_2,f)$. Note that if $t\in [\inf(f),\infty)\setminus  S_+(\varphi_2,f)$, then $\varphi_2\{f\ge t\}=0$, which implies that \eqref{eq:t0} still holds. 
Therefore, with the condition 
of Case 1, we have 
\begin{align*}
\widetilde{\varphi}_1(f)&\ge\widehat{\varphi}_1(f)=\int_{\inf f}^{\sup f} \varphi_1\{f\ge t\}dt+\varphi_1(J)\inf f=\int_{\inf f}^{\sup f} \varphi_1\{f\ge t\}dt
\\&\ge C\int_{\inf f}^{\sup f} \varphi_2\{f\ge t\}dt=C\int_{\inf f}^{\sup f} \varphi_2\{f\ge t\}dt+C\varphi_2(J)\inf f
=C \widehat{\varphi}_2(f)\ge C \widetilde{\varphi}_2(f).
\end{align*}
In consequence, 
$$ \frac{\widetilde{\varphi}_1(f)}{\widetilde{\varphi}_2(f)}\ge C=\inf\limits_{t\in S_+(\varphi_2,f)} \frac{\varphi_1\{f\ge t\}}{\varphi_2\{f\ge t\}}\ge \inf_{A\in\B\text{ with }\varphi_2(A)>0}\frac{\varphi_1(A)}{\varphi_2(A)}.$$
We then obtain 
$$\inf_{f\in \Bd \text{ with }\widetilde{\varphi}_2(f)>0}\frac{\widetilde{\varphi}_1(f)}{\widetilde{\varphi}_2(f)}=\inf_{A\in\B\text{ with }\varphi_2(A)>0}\frac{\varphi_1(A)}{\varphi_2(A)} .   
$$
\vspace{0.1cm}

\noindent Case 2: Suppose that  $\widetilde{\psi}_1$ is a super-Choquet extension  of $\psi_1$, and $\widetilde{\psi}_2$ is a sub-Choquet extension  of $\psi_2$.

In this case, we have $\widetilde{\psi}_1(1_A-1_{A'})=\psi_1(A,A')$ and $\widetilde{\psi}_2(1_A-1_{A'})=\psi_2(A,A')$ for any $(A,A')\in\B_2$, as well as  $\widetilde{\psi}_2(f)\le \widehat{\psi}_2(f)$ and  $\widetilde{\psi}_1(f)\ge \widehat{\psi}_1(f)$ for any $f\in\Bd$. 
This implies the easy direction: $$\inf_{f\in \Bd \text{ with }\widetilde{\psi}_2(f)>0}\frac{\widetilde{\psi}_1(f)}{\widetilde{\psi}_2(f)}\le\inf_{(A,A')\in \B_2\text{ and }\psi_2(A,A')>0}\frac{\psi_1(A,A')}{\psi_2(A,A')}.$$

We shall also focus on the hard direction, for any $f\in \Bd $ with $\widetilde{\psi}_2(f)>0$, let $S_+(\psi_2,f):=\{t>0: \psi_2(f\ge t,f\le-t)>0\}$. Clearly, $\widehat{\psi}_2(f)>0$, and the Lebesgue measure of $S_+(\psi_2,f)$ is positive. Let $$C'=\inf_{(A,A')\in \B_2\text{ and }\psi_2(A,A')>0}\frac{\psi_1(A,A')}{\psi_2(A,A')}$$
Then for any $f\in \Bd $, for any $t\in S_+(\psi_2,f)$, 
\begin{equation}\label{eq:t1}
\psi_1(f\ge t,f\le-t)\ge C' \psi_2(f\ge t,f\le-t).    
\end{equation}
It is clear that for any $t\in\R$, \eqref{eq:t1} still holds. Therefore, 
\begin{align*}
\widetilde{\psi}_1(f)\ge \widehat{\psi}_1(f)&=\int_{0}^{\|f\|_\infty} \psi_1(f\ge t,f\le-t)dt
\\&\ge C'\int_{0}^{\|f\|_\infty} \psi_2(f\ge t,f\le-t)dt
\\&=C' \widehat{\psi}_2(f)\ge C' \widetilde{\psi}_2(f)
\end{align*} 
If $\widetilde{\psi}_2(f)>0$, we immediately obtain 
$$ \frac{\widehat{\psi}_1(f)}{\widehat{\psi}_2(f)}\ge C'$$
which completes the proof in a similar manner as Case 1.

\vspace{0.2cm}

\noindent Case 3:  Suppose that $\varphi_1(J)=\varphi_2(J)=0$, $\varphi_2$ is submodular,   $\widetilde{\varphi}_1$ is a super-Choquet extension  of $\varphi_1$, and $\widetilde{\varphi}_2$ is a one-homogeneous convex extension of $\varphi_2$.

In this case, $\widetilde{\varphi}_2$ is a l.s.c. convex one-homogeneous even function, and  $\widetilde{\varphi}_2(1_A)=\varphi_2(A)$, $\forall A\in\B$. 
Then there exists a convex set $S\ni0$ such that  $\widetilde{\varphi}_2(f)=\sup_{g\in S}\int_J g(x)f(x)dx$.  
Note that $\widetilde{\varphi}_2(-1)=\widetilde{\varphi}_2(1)=\varphi_2(J)=0$ which means that $\sup_{g\in S}\int_J g(x)dx=0=\sup_{g\in S}\int_J (-g(x))dx=-\inf_{g\in S}\int_J g(x)dx$. Thus, $\int_J g(x)dx=0$ for any $g\in S$. 

By Layer Cake representation, we have 
\begin{equation}
f(x) = \int_{\inf (f)}^{\sup(f)} \mathbf{1}_{\{y: f(y) > t\}}(x) \, dt+\inf(f).
\label{eq:LCR-pointwise}
\end{equation}
Then, for any $f\in\Bd$, 
\begin{align*}
\widetilde{\varphi}_2(f)&=\sup_{g\in S}\int_J g(x)f(x)dx  =\sup_{g\in S}\int_J g(x)\Big(\int_{\inf (f)}^{\sup(f)} \mathbf{1}_{\{y: f(y) > t\}}(x) \, dt+\inf(f)\Big)dx 
\\&=\sup_{g\in S}\Big(\int_{\inf (f)}^{\sup(f)} \int_J g(x)\mathbf{1}_{\{y: f(y) > t\}}(x)dx  dt\Big) 
\\&\le \int_{\inf (f)}^{\sup(f)}\sup_{g\in S} \int_J g(x)\mathbf{1}_{\{y: f(y) > t\}}(x)dx
\\&=\int_{\inf (f)}^{\sup(f)} \widetilde{\varphi}_2(\mathbf{1}_{\{y: f(y) > t\}}) dt=\int_{\inf (f)}^{\sup(f)} \varphi_2\{f \ge t\} dt= \widehat{\varphi}_2(f)
\end{align*}
where we used $\int_J g(x) dx=0$, $\forall g\in S$.

Then, we indeed prove that  $\widetilde{\varphi}_2\le \widehat{\varphi}_2$ holds. 
Thus, 
$\widetilde{\varphi}_2$ is actually a sub-Choquet extension of $\varphi_2$, and the result follows from Case 1. 

\vspace{0.2cm}

\noindent Case 4: Suppose that  $\widetilde{\psi}_1$ is a super-Choquet extension  of $\psi_1$, $\psi_2$ is bisubmodular, and $\widetilde{\psi}_2$ is a one-homogeneous convex extension of $\psi_2$. 

In this case, we shall plug the following Layer Cake representation
\begin{equation}
f(x) = \int_{0}^{\|f\|_\infty} \big(\mathbf{1}_{\{y: f(y) \ge t\}}(x)  - \mathbf{1}_{\{y: f(y) \le- t\}}(x)\big)  dt.
\label{eq:LCR-pointwise-pair}
\end{equation}
into the expression of $\widetilde{\psi}_2(f)$, and then we obtain
\begin{align*}
\widetilde{\psi}_2(f)&=\sup_{g\in S}\int_J g(x)f(x)dx  \\&=\sup_{g\in S}\int_J g(x)\Big(\int_{0}^{\|f\|_\infty} \big(\mathbf{1}_{\{y: f(y) \ge t\}}(x)  - \mathbf{1}_{\{y: f(y) \le- t\}}(x)\big)  dt\Big)dx 
\\&=\sup_{g\in S}\Big(\int_{0}^{\|f\|_\infty} \int_J g(x)\big(\mathbf{1}_{\{y: f(y) \ge t\}}(x)  - \mathbf{1}_{\{y: f(y) \le- t\}}(x)\big)dx  dt\Big) 
\\&\le \int_{0}^{\|f\|_\infty} \Big(\sup_{g\in S}\int_J g(x)\big(\mathbf{1}_{\{y: f(y) \ge t\}}(x)  - \mathbf{1}_{\{y: f(y) \le- t\}}(x)\big)dx \Big) dt 
\\&=\int_{0}^{\|f\|_\infty} \widetilde{\psi}_2(\mathbf{1}_{\{y: f(y) \ge t\}}-\mathbf{1}_{\{y: f(y) \le- t\}}) dt=\int_{0}^{\|f\|_\infty} \psi_2(f \ge t,f \le- t) dt= \widehat{\psi}_2(f)
\end{align*}
The desired conclusion then follows from Case 2.

Let $\Bd_+=\{f\in\Bd:f(x)\ge0,\forall x\in J\}$. Denote by $\bm{\varPsi}_+(\B)$ the collection of all $\varphi:\B\to\R$ such that for any  $f\in\Bd_+$, the function $t\mapsto \varphi\{f\ge t\}$ is Lebesgue integrable on the interval $[\inf (f),\infty)$. 
\begin{prop}
Given $\varphi_1,\varphi_2\in\bm{\varPsi}_+(\B)$ which are nonnegative, we have 
\begin{equation}\label{eq:Choquet-identity+}
\inf_{A\in\B\text{ with }\varphi_2(A)>0}\frac{\varphi_1(A)}{\varphi_2(A)}   =\inf_{f\in \Bd_+\text{ with }\widetilde{\varphi}_2(f)>0}\frac{\widetilde{\varphi}_1(f)}{\widetilde{\varphi}_2(f)}   
\end{equation}
where $\widetilde{\varphi}_1$ is a super-Choquet extension of $\varphi_1$, and $\widetilde{\varphi}_2$ is a sub-Choquet extension of $\varphi_2$, that is, $\widetilde{\varphi}_1(1_A)=\varphi_1(A)$ and $\widetilde{\varphi}_2(1_A)=\varphi_2(A)$ for any $A\in\B$,  $\widetilde{\varphi}_1(f)\ge \widehat{\varphi}_1(f)$ and $\widetilde{\varphi}_2(f)\le \widehat{\varphi}_2(f)$ for any $f\in\Bd_+$. 
\end{prop}

\begin{proof}
It is clear that 
\begin{align*}
 \inf_{f\in \Bd_+ \text{ with }\widehat{\varphi}_2(f)>0}\frac{\widehat{\varphi}_1(f)}{\widehat{\varphi}_2(f)}
\le \inf_{{\bf1}_A\in \Bd_+ \text{ with }\widehat{\varphi}_2({\bf1}_A)>0}\frac{\widehat{\varphi}_1({\bf1}_A)}{\widehat{\varphi}_2({\bf1}_A)} =\inf_{A\in\B\text{ with }\varphi_2(A)>0}\frac{\varphi_1(A)}{\varphi_2(A)}.
\end{align*}
For the other direction, we start from the inequality \eqref{eq:t0} and derive \begin{align*}
\widetilde{\varphi}_1(f)&\ge\widehat{\varphi}_1(f)=\int_{\inf f}^{\sup f} \varphi_1\{f\ge t\}dt+\varphi_1(J)\inf f
\\&\ge C\int_{\inf f}^{\sup f} \varphi_2\{f\ge t\}dt+C\varphi_2(J)\inf f
\\&=C \widehat{\varphi}_2(f)\ge C \widetilde{\varphi}_2(f)
\end{align*}
whenever  $\inf f\ge0$, i.e., $f\in\Bd_+$. Using a similar procedure used in the proof of Theorem \ref{thm:tilde-fg-equal}, 
we have 
$$\inf_{f\in \Bd_+ \text{ with }\widehat{\varphi}_2(f)>0}\frac{\widehat{\varphi}_1(f)}{\widehat{\varphi}_2(f)}=\inf_{A\in\B\text{ with }\varphi_2(A)>0}\frac{\varphi_1(A)}{\varphi_2(A)}.    
$$
Then, for any $f\in\Bd_+$, we have 
\begin{align*}
\widetilde{\varphi}_2(f)&=\sup_{g\in S}\int_J g(x)f(x)dx  =\sup_{g\in S}\int_J g(x)\Big(\int_{0}^{\sup(f)} \mathbf{1}_{\{y: f(y) > t\}}(x) \, dt\Big)dx 
\\&=\sup_{g\in S}\Big(\int_{0}^{\sup(f)} \int_J g(x)\mathbf{1}_{\{y: f(y) > t\}}(x)dx  dt\Big) 
\\&\le \int_{0}^{\sup(f)}\sup_{g\in S} \int_J g(x)\mathbf{1}_{\{y: f(y) > t\}}(x)dx
\\&=\int_{0}^{\sup(f)} \widetilde{\varphi}_2(\mathbf{1}_{\{y: f(y) \ge t\}}) dt=\int_{0}^{\sup(f)} \varphi_2\{f \ge t\} dt = \widehat{\varphi}_2(f)
\end{align*}
\end{proof}

We can also work on  $L^\infty(J) _+:=\{f\in L^\infty(J):f(x)\ge 0,\forall x\in J\text{ a.e.}\}$ instead of $\Bd_+$, and use $\essinf$ instead of $\inf$. 

Using a completely similar proof,  we have the following `sup'-analogue of Theorem \ref{thm:tilde-fg-equal}.
\begin{theorem}[A dual version of Theorem \ref{thm:tilde-fg-equal}] \label{th:dual-main}
Given $\varphi_1,\varphi_2\in\bm{\varPsi}(\B)$ which are nonnegative and satisfy $\varphi_1(J)=\varphi_2(J)=0$, we have 
\begin{equation}\label{eq:Choquet-identity/sup}
\sup_{A\in\B\text{ with }\varphi_2(A)>0}\frac{\varphi_1(A)}{\varphi_2(A)} =\sup_{f\in \Bd\text{ with }\widehat{\varphi}_2(f)>0}\frac{\widehat{\varphi}_1(f)}{\widehat{\varphi}_2(f)}   =\sup_{f\in \Bd\text{ with }\widetilde{\varphi}_2(f)>0}\frac{\widetilde{\varphi}_1(f)}{\widetilde{\varphi}_2(f)} .  
\end{equation}
where $\widetilde{\varphi}_1$ is a sub-Choquet extension  of $\varphi_1$ (or a one-homogeneous convex extension of $\varphi_1$ if $\varphi_1$ is submodular), and $\widetilde{\varphi}_2$ is a super-Choquet extension  of $\varphi_2$ . 

Let $\psi_1, \psi_2 \in \bm{\varPsi}(\B_2)$, and assume that they are both nonnegative. 
Then \begin{equation}\label{eq:Choquet-pair-identity/sup}
\sup_{(A,A')\in \B_2\text{ and }\psi_2(A,A')>0}\frac{\psi_1(A,A')}{\psi_2(A,A')}=\sup_{f\in \Bd\text{ with }\widehat{\psi}_2(f)>0}\frac{\widehat{\psi}_1(f)}{\widehat{\psi}_2(f)}=\sup_{f\in \Bd\text{ with }\widetilde{\psi}_2(f)>0}\frac{\widetilde{\psi}_1(f)}{\widetilde{\psi}_2(f)}   
\end{equation}
where $\widetilde{\psi}_1$ is a sub-Choquet extension  of $\psi_1$  (or a one-homogeneous convex extension of $\psi_1$ if $\psi_1$ is bisubmodular), and $\widetilde{\psi}_2$ is a super-Choquet extension of $\psi_2$.
\end{theorem}

\subsection{Proof of Theorem \ref{thm:critical-inequality}}

Hereafter we simply use $f\mapsto f^r$ to represent the Mazur map 
defined by 
\begin{equation}\label{eq:Mazur-map}
f^r(x):=|f(x)|^r\mathrm{sign}(f(x)),\;\;\forall x\in J    
\end{equation}
where $$\mathrm{sign}(t)=\begin{cases}
1,&\text{ if }t>0\\
0,&\text{ if } t=0\\
-1,&\text{ if } t<0\\
\end{cases}$$ is the standard sign function.

\noindent Case 1. Using the original Choquet integral under general setting  

Step 1.1. Prove the inequality 
\begin{equation}\label{ew:p-monotonicity<=r}
\frac{\Phi_p(f^t)}{ \big\llbracket f^t\big\rrbracket_{q}}\le t\frac{\Phi_{ps}(f)}{\big\llbracket f\big\rrbracket_{qt}}
\end{equation} 
for arbitrary $f\in\Bd$.

Proof for Step 1.1:  
Given $p,q,s,t\ge1$ with $ps'=qt'$,  where $s'$ and $t'$ are the H\"older conjugates of $s$ and $t$, respectively. 
If $t=1$, then $t'=\infty$ and $ps'=\infty$, which means that $p=\infty$ or $s'=\infty$, and thus $p=ps$, yielding that \eqref{ew:p-monotonicity<=r} reduces to a trivial identity. 
Hence, without loss of generality, we may assume $1<t<\infty$. 

Since in Case 1 we work with the original Choquet integral,  by noting that $\frac{d}{dx}|x|^{t-1}x=t|x|^{t-1}$, we have the following inequality 
\begin{align}
\int_E\big(\widehat{\varphi}_e(f^t)\big)^pd\eta(e)&=\int_E\Big(\int_{\inf f}^\infty \varphi_e(f^t\ge x)dx+\varphi_e(X)\inf f^t\Big)^pd\eta(e) \notag
\\(\text{taking }x=|y|^{t-1}y)&=\int_E\Big(t\int_{\inf f}^\infty |y|^{t-1}\varphi_e(f^t\ge y^t)dy\Big)^pd\eta(e)  \notag
\\&=t^p\int_E\Big(\int_{\inf f}^\infty |y|^{t-1}\varphi_e(f\ge y)dy\Big)^pd\eta(e) \notag 
\\&=t^p\int_E\Big(\int_{\inf f}^{\sup f} |y|^{t-1}\varphi_e(f\ge y)dy\Big)^pd\eta(e) \label{eq:=r^p..} 
\\&\le t^p\int_E|f|_{\varphi_e}^{(t-1)p}\Big(\int_{\inf f}^{\sup f} \varphi_e(f\ge y)dy\Big)^pd\eta(e) \notag
\\&= t^p\int_E|f|_{\varphi_e}^{(t-1)p}\big(\widehat{\varphi}_e(f)\big)^pd\eta(e) \notag
\\&\le t^p\left(\int_E|f|_{\varphi_e}^{(t-1)ps'}d\eta(e)\right)^{\frac{1}{s'}}\left(\int_E\big((\widehat{\varphi}_e(f)\big)^{ps}d\eta(e)\right)^{\frac{1}{s}}\notag
\\&=t^p\left(\int_E|f|_{\varphi_e}^{qt}d\eta(e)\right)^{\frac{1}{s'}}\left(\int_E\big((\widehat{\varphi}_e(f)\big)^{ps}d\eta(e)\right)^{\frac{1}{s}}\notag
\end{align}
where we have set 
$|f|_{\varphi_e}:=\max\{\sup\{t\in \R: \varphi_e\{f\ge t\}> 0\},-\inf\{t\in \R: \varphi_e\{f\ge t\}> 0\}\}$ as the  infinity norm of $f$ with respect to $\varphi_e$. 

Then we obtain 
\[ \frac{\Phi_p(f^t)}{ \big\llbracket f^t\big\rrbracket_{q}}=\frac{\left(\int_E\big(\widehat{\varphi}_e(f^t)\big)^pd\eta(e)\right)^{\frac1p}}{\left(\int_E|f^t|_{\varphi_e}^{q}d\eta(e)\right)^{\frac{1}{q}}} \le t\frac{\left(\int_E\big((\widehat{\varphi}_e(f)\big)^{ps}d\eta(e)\right)^{\frac{1}{ps}}}{\left(\int_E|f|_{\varphi_e}^{qt}d\eta(e)\right)^{\frac{1}{qt}}}=t\frac{\Phi_{ps}(f)}{\big\llbracket f\big\rrbracket_{qt}}
\]
where we have used the relation $\frac1q-\frac{1}{ps'}=\frac{1}{qt}$ that is equivalent to the assumption $ps'=qt'$. This finishes the proof.

Step 1.2. Use the min-max principle to derive the inequality \eqref{eq:c-monotone1}.

Proof for Step 1.2: Observe that according to  the assumption on $\mathcal{Y}$, every $t>0$ induces a bijection $\mathbf{t}:Y\to Y^t$ defined by $\mathbf{t}(f)= f^t$, and also induces a bijection $\tilde{\mathbf{t}}:\mathcal{Y}\to \mathcal{Y}^t$ defined as  $\tilde{\mathbf{t}}(Y)= Y^t$. 
Taking the min-max procedure on both sides of \eqref{ew:p-monotonicity<=r}, we obtain 
\begin{align*}
c(\Phi_p,\llbracket\cdot\rrbracket_{ q })&:=\inf\limits_{Y\in \mathcal{Y}}\sup\limits_{f\in Y} \frac{\Phi_p(f)}{\llbracket f \rrbracket_{q}}
\\&
=\inf\limits_{Y^t\in \mathcal{Y}}\sup\limits_{g\in Y^t} \frac{\Phi_p(g)}{\llbracket g\rrbracket_{q}}
\\&=\inf\limits_{Y\in \mathcal{Y}}\sup\limits_{f\in Y} \frac{\Phi_p(f^t)}{\llbracket f^t\rrbracket_{q}}
\\ \text{by \eqref{ew:p-monotonicity<=r}}&\le\inf\limits_{Y\in \mathcal{Y}}\sup\limits_{f\in Y}  t\frac{\Phi_{ps}(f)}{\llbracket f\rrbracket_{qt}}=t\cdot  c(\Phi_{ps},\llbracket\cdot\rrbracket_{qt})
\end{align*}
where $Y^t=\{f^t:f\in Y\}\in \mathcal{Y}$. 
This completes the proof of \eqref{eq:c-monotone1} in Case 1. 

Taking $q=p$ and $s=t$, we simply  have $c(\Phi_p,\llbracket \cdot\rrbracket_{p})\le t\cdot  c(\Phi_{pt},\llbracket \cdot\rrbracket_{pt})$, which implies that $p\cdot c(\Phi_p,\llbracket \cdot\rrbracket_{p})$ is increasing with respect to $p$. 

\begin{remark}\label{rem:4-3parameters}
Note that the four parameters $p,q,t,s$ are not independent as they satisfy the relation $ps'=qt'$, which can thus be reduced to three independent parameters. 
Equivalently, \eqref{eq:c-monotone1} can be written in the following form with only three parameters:
For any $p,q,r\ge1$ with  
$(1-\frac1r)pq'\ge1$, where $q'$ is the H\"older conjugate of $q$, we have the following inequality 
\begin{equation}\label{eq:c-monotone12}
c(\Phi_p,\llbracket \cdot\rrbracket_{  (1-\frac1r)pq' })\le r\cdot  c(\Phi_{pq},\llbracket \cdot\rrbracket_{(r-1)pq'}).    
\end{equation}

\end{remark}

Finally, we prove that  $\eta(E)^{\frac1q-\frac1p}c(\Phi_p,\llbracket \cdot\rrbracket_{q})$  is increasing with respect to $p$ and decreasing with respect to $q$. 
We need the following elementary and widely known lemma. 

\begin{lemma}\label{lem:increase-mean-power}
For any probability measure $\mathbb{P}$ on $X$, the function $p\mapsto (\int_X |f(x)|^pd\mathbb{P}(x))^{\frac1p}$ is increasing. 
\end{lemma}

By Lemma \ref{lem:increase-mean-power}, $$\eta(E)^{-\frac1p}\Phi_p(f)=\Big(\frac{\int_E \big(\widehat{\varphi}_e(f)\big)^pd\eta(e)}{\eta(E)}\Big)^{\frac1p}$$
increases with respect to $p$, and 
$$\eta(E)^{-\frac1q}\big\llbracket f \big\rrbracket_q =\Big(\frac{\int_E|f|_{\varphi_e}^qd\eta(e)}{\eta(E)}\Big)^{\frac1q}$$
increases with respect to $q$. 
A min-max procedure similar to Step 1.2 concludes the monotonicity property of $\eta(E)^{\frac1q-\frac1p}c(\Phi_p,\llbracket \cdot\rrbracket_{q})$. 

\vspace{0.2cm}

\noindent Case 2. using the disjoint-pair  Choquet integral under general setting  

In this case, we have the following inequality:
\begin{align*}
\int_E\big(\widehat{\psi}_e(f^r)\big)^pd\eta(e)&=\int_E\Big(\int_0^\infty \psi_e(f^r\ge t,f^r\le -t)dt\Big)^pd\eta(e)
\\(\text{taking }t=s^r)&=\int_E\Big(r\int_0^\infty s^{r-1}\psi_e(f^r\ge s^r,f^r\le -s^r)ds\Big)^pd\eta(e) \\&=r^p\int_E\Big(\int_0^\infty s^{r-1}\psi_e(f\ge s,f\le -s)ds\Big)^pd\eta(e) 
\\&=r^p\int_E\Big(\int_0^{\|f\|_{\psi_e}} s^{r-1}\psi_e(f\ge s,f\le -s)ds\Big)^pd\eta(e)  
\\&\le r^p\int_E\|f\|_{\psi_e}^{(r-1)p}\Big(\int_0^{\|f\|_{\psi_e}} \psi_e(f\ge s,f\le -s)ds\Big)^pd\eta(e)
\\&= r^p\int_E\|f\|_{\psi_e}^{(r-1)p}\big(\widehat{\psi}_e(f)\big)^pd\eta(e)
\\&\le r^p\left(\int_E\|f\|_{\psi_e}^{(r-1)pq'}d\eta(e)\right)^{\frac{1}{q'}}\left(\int_E\big((\widehat{\psi}_e(f)\big)^{pq}d\eta(e)\right)^{\frac{1}{q}}.
\end{align*}
The remaining proof is similar to that of Case 1, and concerning Remark \ref{rem:4-3parameters}, we omit the details.

\subsection{Proof of Theorem \ref{thm:critical-inequality2}}

Case 1. 
We further suppose that $J$, $E$ and $\{\varphi_e\}_{e\in E}$  satisfy Definitions \ref{def:hyper-phi-concentrate} and  \ref{def:hyper-degree-related}, and prove  \eqref{eq:c-monotone1} by using $\widetilde{\llbracket \cdot\rrbracket}$ instead of $\llbracket \cdot\rrbracket$.

\noindent Proof of Case 1: 
Note that $J$ and $E$ are degree-related, and for each $e\in E$, $\varphi_e$ is concentrated on $e$. 
In this case, the infinity norm of $f$ with respect to $\varphi_e$, i.e.,  $|f|_{\varphi_e}:=\inf \big\{t>0: \varphi_e\{f\ge s\}=\varphi_e\{f\ge -s\}=0,\forall s\ge t\big\}$ is less than or equal to $\max_{x\in e_J}|f(x)|$. Therefore, $|f^t|_{\varphi_e}^{q}=|f|_{\varphi_e}^{qt}\le \max_{x\in e_J}|f(x)|^{qt}\le \sum_{x\in e_J}|f(x)|^{qt}$ and $$\int_E|f|_{\varphi_e}^{qt}d\eta(e)\le \int_E\sum_{x\in e_J}|f(x)|^{qt}d\eta(e)\le\int_J\deg(x)|f(x)|^{qt}d\mu(x).$$  
This implies $c(\Phi_p,\widetilde{\llbracket \cdot\rrbracket}_{q})\le t\cdot  c(\Phi_{ps},\widetilde{\llbracket \cdot\rrbracket}_{qt})$. Below, we show an approach to enhance the degree term.

\begin{prop}\label{pro:phi-rKcpq}
Suppose that $C_{JE}:=\sup_{e\in E}(|e_J|-1)<\infty$, and let $K_{p,q}=\big(\frac{C_{JE}^{p-1}}{2}\min \{2,C_{JE}\}\big)^{\frac{1}{pq'}}$. Then
$$c(\Phi_p,\widetilde{\llbracket \cdot\rrbracket}_{  (1-\frac1r)pq' })\le r   K_{p,q}\cdot  c(\Phi_{pq},\widetilde{\llbracket \cdot\rrbracket}_{(r-1)pq'})$$ 
\end{prop}

\begin{proof}
By condition, for any $e\in E$, $e_J:=\{x\in J:e\text{ corresponds to }x\}$ is a finite subset of $J$. 
Thus, $\{\{f\ge s\}\cap e_J:s\in \R\}$ and 
$\{\varphi_e\{f\ge s\}:s\in \R\}$ are finite sets. Therefore, there exist $-\infty<s_0 < s_1< \cdots< s_m< +\infty
$ such that both $\{f\ge s\}\cap e_J$ and $\varphi_e\{f\ge s\}$ are constant on each interval $(s_{i-1},s_i)$, $i=1,\cdots,m$. Clearly, $\varphi_e\{f\ge s\}=0$ when $s> s_m$ or $s\le s_0$. 
We further assume that the $m$ mentioned above is the smallest positive integer satisfying the aforementioned property. 
Now, to distinguish these $s_1,\cdots,s_m$ for different $e\in E$, in the sequel, we write $s_{e,i}$ instead of $s_i$, and use $I_e$ to denote their indexes $\{1,\cdots,m\}$. \;\;
Denote by
\begin{equation}\label{eq:e_phi}
e_\varphi=\{x\in e_J:\varphi_e(S\setminus x)\ne \varphi_e(S)\text{ for some }S\subset e_J\}.    
\end{equation}
We shall prove the following argument. 

Argument 1: Fixed a function $f\in\Bd$, for any $e\in E$, we have
$$\{s_{e,1},\cdots,s_{e,m}\}\subset \{f(x):x\in e_\varphi\}\subset \{f(x):x\in e_J\} $$
and 
$$e_{\varphi,f}:=\{x\in e_\varphi:f(x)=s_{e,i}\text{ for some }i\}\subset e_\varphi\subset e_J.$$
Proof of Argument 1: It suffices to prove the first inclusion relation. Suppose the contrary, that there exists $s_{e,i}$ such that $s_{e,i}\ne f(x)$ for any $x\in e_\varphi$. 
Suppose $f^{-1}(s_{e,i})\cap e_J=\{x_{e,1},\cdots,x_{e,k}\}$. Then $\{x_{e,1},\cdots,x_{e,k}\}\cap e_\varphi=\emptyset$, and there exists sufficiently small $\varepsilon>0$ such that $\{f\ge s_{e,i}+\varepsilon\}\cap e_J=\{f\ge s_{e,i}\}\cap e_J\setminus\{x_{e,1},\cdots,x_{e,k}\}$ and 
\begin{equation}\label{eq:s_e,i-ne-+}
\varphi_e\{f\ge s_{e,i}+\varepsilon\}\ne \varphi_e\{f\ge s_{e,i}\}    
\end{equation} (for instance, we can simply take $0<\varepsilon<\frac12 (s_{e,i+1}-s_{e,i})$ to guarantee‌ this property). 

Since $x_{e,1}\not\in e_\varphi$, by \eqref{eq:e_phi}, for any $S\subset e_J$, $\varphi_e(S)=\varphi_e(S\setminus x_{e,1})$. In particular, take $S=\{f\ge s_{e,i}\}\cap e_J$ and then we derive that 
$$\varphi_e\{f\ge s_{e,i}\}=\varphi_e(\{f\ge s_{e,i}\}\cap e_J)=\varphi_e(\{f\ge s_{e,i}\}\cap e_J\setminus x_{e,1})=\varphi_e(\{f\ge s_{e,i}\} \setminus x_{e,1}).$$ 

Since $x_{e,2}\not\in e_\varphi$, for any $S\subset e$, $\varphi_e(S)=\varphi_e(S\setminus x_{e,2})$. We can then take $S=\{f\ge s_{e,i}\}\cap e_J\setminus x_{e,1}$, and further derive $\varphi_e(\{f\ge s_{e,i}\} \setminus x_{e,1})=\varphi_e(\{f\ge s_{e,i}\} \setminus \{x_{e,1},x_{e,2}\})$.

Repeating the above process, we finally obtain 
\begin{align*}
\varphi\{f\ge s_{e,i}\}&=\varphi\{f\ge s_{e,i}\}
\\&=\varphi(\{f\ge s_{e,i}\} \setminus \{x_{e,1}\})
\\&=\varphi(\{f\ge s_{e,i}\} \setminus \{x_{e,1},x_{e,2}\})
\\&=\cdots
\\&=    \varphi(\{f\ge s_{e,i}\} \setminus \{x_{e,1},x_{e,2},\cdots,x_{e,k}\})
\\&=\varphi\{f\ge s_{e,i}+\varepsilon\}
\end{align*}
which contradicts \eqref{eq:s_e,i-ne-+}. The proof of Argument 1 is then completed.

\vspace{0.2cm}

Back to the proof of Case 1. By Argument 1, $e_{\varphi,f}\subset e_J$. 
Then, by Eq.~\eqref{eq:=r^p..}, 
\begin{align}
&\int_E\big(\widehat{\varphi}_e(f^r)\big)^pd\eta(e)\notag
\\=~& r^p\int_E\Big(\int_{-|f|_{\varphi_e}}^{|f|_{\varphi_e}} |s|^{r-1}\varphi_e(f\ge s)ds\Big)^pd\eta(e)   \notag
\\=~& \int_E\Big(\sum_{i\in I_e} \int_{s_{e,i-1}}^{s_{e,i}}r|s|^{r-1}\varphi_e(f\ge s)ds\Big)^pd\eta(e) \notag
\\=~&\int_E\Big(\sum_{i\in I_e} (|s_{e,i}|^{r-1}s_{e,i}-|s_{e,i-1}|^{r-1}s_{e,i-1})\varphi_e(f\ge s_{e,i-1})\Big)^pd\eta(e)  \label{eq:s^r-s^r}
\\ \le~&\int_E\left(\sum_{i\in I_e}r(s_{e,i}-s_{e,i-1})\big(\frac{|s_{e,i}|^r+|s_{e,i-1}|^r}{2}\big)^{\frac{1}{r'}}\varphi_e(f\ge s_{e,i-1})\right)^pd\eta(e)  \notag
\\ \le~&
r^p\left(\int_E\Big(\sum_{i\in I_e}(s_{e,i}-s_{e,i-1})^q\varphi_e^q(f\ge s_{e,i-1})\Big)^pd\eta(e)\right)^{\frac1q}\times \label{eq:Holder-twice}
\\&~~~~~~~~~~~~~~\times
\left(\int_E\Big(\sum_{i\in I_e} \big(\frac{|s_{e,i}|^r+|s_{e,i-1}|^r}{2}\big)^{\frac{q'}{r'}}\Big)^pd\eta(e)\right)^{\frac{1}{q'}}  \notag
\\ \le~&
r^p\left(\int_E\Big(\sum_{i\in I_e}(s_{e,i}-s_{e,i-1})\varphi_e(f\ge s_{e,i-1})\Big)^{qp}d\eta(e)\right)^{\frac1q}\times\label{eq:constantC}
\\&~~~~~~~~~~~~~~\times
\left(\int_E|I_e|^{p-1}\sum_{i\in I_e} \big(\frac{|s_{e,i}|^r+|s_{e,i-1}|^r}{2}\big)^{\frac{q'}{r'}p}d\eta(e)\right)^{\frac{1}{q'}} \notag
\\ \le~&
r^p\left(\int_E\Big(\widehat{\varphi}_e(f)\Big)^{qp}d\eta(e)\right)^{\frac1q}
\left(\int_E|I_e|^{p-1}\sum_{i\in I_e} \frac{|s_{e,i}|^{(r-1)q'p}+|s_{e,i-1}|^{(r-1)q'p}}{2}d\eta(e)\right)^{\frac{1}{q'}} \notag
\\ \le ~&
r^p\sup_{e\in E}\big(\frac{|I_e|^{p-1}}{2}\min \{2,|I_e|\}\big)^{\frac{1}{q'}}\left(\int_E\Big(\widehat{\varphi}_e(f)\Big)^{pq}d\eta(e)\right)^{\frac1q}
\left(\int_E\sum_{i\in I_e\cup\{0\}} |s_{e,i}|^{(r-1)q'p}d\eta(e)\right)^{\frac{1}{q'}} \label{eq:I_ecup0}
\\ \le ~&
r^p \big(\frac{C_{JE}^{p-1}}{2}\min \{2,C_{JE}\}\big)^{\frac{1}{q'}}\left(\int_E\Big(\widehat{\varphi}_e(f)\Big)^{pq}d\eta(e)\right)^{\frac1q}
\left(\int_E\sum_{x\in e_{\varphi,f}} |f(x)|^{(r-1)q'p}d\eta(e)\right)^{\frac{1}{q'}} \notag
\\ \le ~&r^p K_{p,q}^p\left(\int_E\Big(\widehat{\varphi}_e(f)\Big)^{pq}d\eta(e)\right)^{\frac1q}
\left(\int_J \deg(x)|f(x)|^{(r-1)q'p}d\mu(x)\right)^{\frac{1}{q'}} \notag
\end{align}
where the second-to-last inequality makes use of the fact that $|I_e|\le |e_J|-1\le C_{JE}$, the inequality \eqref{eq:Holder-twice} uses the H\"older inequality twice via 
$$ \int \Big(\sum a_ib_i\Big)^p\le \int (\sum_ia_i^q)^{\frac pq}(\sum_ib_i^{q'})^{\frac {p}{q'}}\le \left(\int (\sum_ia_i^q)^p\right)^{\frac1q}\left(\int (\sum_ib_i^{q'})^p\right)^{\frac{1}{q'}}.$$
The inequality \eqref{eq:constantC} is due to the elementary inequalities that $\sum_{i=1}^n a_i^q\le (\sum_{i=1}^n a_i)^q$ and $(\sum_{i=1}^n a_i)^p\le n^{p-1}\sum_{i=1}^n a_i^p $ for any $p,q\ge 1$ and non-negative $a_i$. 
We shall also explain the terms in \eqref{eq:I_ecup0}. In fact, if $|I_e|=m\ge 2$, i.e., $I_e=\{1,\cdots,m\}$, then
\begin{align*}
\sum_{i\in I_e} \frac{|s_{e,i}|^{(r-1)q'p}+|s_{e,i-1}|^{(r-1)q'p}}{2} &=\frac{|s_{e,0}|^{(r-1)q'p}+|s_{e,m}|^{(r-1)q'p}}{2}+\sum_{i=1}^{m-1}|s_{e,i}|^{(r-1)q'p}
\\&\le \sum_{i\in I_e\cup\{0\}} |s_{e,i}|^{(r-1)q'p}.   
\end{align*}
 While, if $|I_e|=1$, i.e., $I_e=\{1\}$, then 
 $$\sum_{i\in I_e} \frac{|s_{e,i}|^{(r-1)q'p}+|s_{e,i-1}|^{(r-1)q'p}}{2} =\frac{|s_{e,0}|^{(r-1)q'p}+|s_{e,1}|^{(r-1)q'p}}{2}=\frac{1}{2}\sum_{i\in I_e\cup\{0\}} |s_{e,i}|^{(r-1)q'p}.$$
In any case, the term
$$\sum_{i\in I_e} \frac{|s_{e,i}|^{(r-1)q'p}+|s_{e,i-1}|^{(r-1)q'p}}{2}\le \frac{\min\{2,|I_e|\}}{2}\sum_{i\in I_e\cup\{0\}} |s_{e,i}|^{(r-1)q'p} $$
which implies \eqref{eq:I_ecup0}. 
Consequently, by the whole inequalities derived above, we obtain 
\begin{align*}
\frac{\Phi_p(f^r)}{\widetilde{\llbracket f^r\rrbracket}_{(1-\frac1r)pq'}}&=\frac{\left(\int_E\big(\widehat{\varphi}_e(f^r)\big)^pd\eta(e)\right)^{\frac1p}}{\left(\int_J\deg(x)|f^r(x)|^{(1-\frac1r)pq'}d\mu(x)\right)^{\frac{1}{(1-\frac1r)pq'}}} 
\\ &\le r \big(\frac{C_{JE}^{p-1}}{2}\min \{2,C_{JE}\}\big)^{\frac{1}{pq'}} \frac{\left(\int_E\big((\widehat{\varphi}_e(f)\big)^{pq}d\eta(e)\right)^{\frac{1}{pq}}}{\left(\int_J \deg(x)|f(x)|^{(r-1)q'p}dx\right)^{\frac{1}{(r-1)pq'}}}
\\&=rK_{p,q}\frac{\Phi_{pq}(f)}{\widetilde{\llbracket f\rrbracket}_{(r-1)pq'}}    
\end{align*}
whenever $(1-\frac1r)pq'\ge1$. 
Then, by a min-max procedure, we obtain the inequality
$$c(\Phi_p,\widetilde{\llbracket \cdot\rrbracket}_{  (1-\frac1r)pq' })\le r K_{p,q}\cdot   c(\Phi_{pq},\widetilde{\llbracket \cdot\rrbracket}_{(r-1)pq'})$$ 
which completes the proof.
\end{proof}


\begin{remark}
If we assume that  $C_{x}:=\sup\limits_{e\in E:x\in e_J} |I_e|^{p-1}<\infty$ for any $x\in J$, then, after the inequality \eqref{eq:constantC}, we can use the following alternative procedure:

Taking $\widetilde{\deg} (x)=\deg (x)C_x$ for any $x\in J$, and using $\widetilde{\deg}$ instead of $\deg$ in the definition of $\widetilde{\llbracket \cdot\rrbracket}$, we derive $c(\Phi_p,\widetilde{\llbracket \cdot\rrbracket}_{  (1-\frac1r)pq' })\le r \cdot  c(\Phi_{pq},\widetilde{\llbracket \cdot\rrbracket}_{(r-1)pq'}) $.
\end{remark}

\begin{remark}
If $C_{JE}=1$, then we can take $\widetilde{\deg} (x)=\frac12\deg (x)$ for any $x\in J$, and use $\widetilde{\deg}$ instead of $\deg$ in the definition of $\widetilde{\llbracket \cdot\rrbracket}$. In this case, the derived inequality 
$c(\Phi_p,\widetilde{\llbracket \cdot\rrbracket}_{  (1-\frac1r)pq' })\le r \cdot  c(\Phi_{pq},\widetilde{\llbracket \cdot\rrbracket}_{(r-1)pq'}) $ improves 
the related result in Theorem \ref{thm:critical-inequality2}.
\end{remark}

\begin{remark}
Suppose that  $J$ and $E$ form a finitely related structure, and for any $e\in E$, $\varphi_e$ is concentrated on $e$. 
In this case, we can simply take a sub-degree function defined as $\widetilde{\deg}_\varphi(x)=\eta\{e\in E: \varphi_e(S)\ne\varphi_e(S\setminus \{x\})\text{ for some }S\subset e_J\text{ with }S,S\setminus \{x\}\in \B\}$.       
\end{remark}

\begin{prop}\label{pro:phi-cpq21-r}
Suppose that $C_{JE}:=\sup\limits_{e\in E}(|e_J|-1)<\infty$ and $C_{\Phi}:=\sup\limits_{e\in E}\|\varphi_e\|_\infty<\infty$. Then, for any degree and its related norm $\widetilde{\llbracket \cdot\rrbracket}$, 
$$c(\Phi_p,\widetilde{\llbracket \cdot\rrbracket}_{q})\ge  C_{JE}^{\frac1p-r}(2C_{\Phi})^{1-r}\big(c(\Phi_{pr},\widetilde{\llbracket \cdot\rrbracket}_{qr})\big)^r.$$ 
\end{prop}

\begin{proof}
By Eqs.~\eqref{eq:=r^p..} 
 and \eqref{eq:s^r-s^r}, $\int_E\big(\widehat{\varphi}_e(f^r)\big)^pd\eta(e)$ equals
\begin{align*}
&\int_E\Big(\sum_{i\in I_e} (|s_{e,i}|^{r-1}s_{e,i}-|s_{e,i-1}|^{r-1}s_{e,i-1})\varphi_e(f\ge s_{e,i-1})\Big)^pd\eta(e)
\\ \ge~& \int_E \sum_{i\in I_e} (|s_{e,i}|^{r-1}s_{e,i}-|s_{e,i-1}|^{r-1}s_{e,i-1})^p\varphi_e(f\ge s_{e,i-1}) ^pd\eta(e)
\\ \ge~&\int_E\sum_{i\in I_e} (s_{e,i}-s_{e,i-1})^p\left(\frac{|s_{e,i}|^r+|s_{e,i-1}|^r}{2}\right)^{(1-\frac1r)p}\varphi_e(f\ge s_{e,i-1})^pd\eta(e)
\\ \ge~&\int_E\frac{1}{ \|\varphi_e\|_\infty^{rp-p}}\sum_{i\in I_e} (s_{e,i}-s_{e,i-1})^p\left(\frac{s_{e,i}-s_{e,i-1}}{2}\right)^{r(1-\frac1r)p}\varphi_e(f\ge s_{e,i-1})^{rp}d\eta(e)
\\ \ge~&\int_E\frac{2^{(1-r)p}}{ \|\varphi_e\|_\infty^{rp-p}}\sum_{i\in I_e} (s_{e,i}-s_{e,i-1})^{rp}\varphi_e(f\ge s_{e,i-1})^{rp}d\eta(e)
\\ \ge~&\int_E \frac{2^{(1-r)p}}{ \|\varphi_e\|_\infty^{rp-p}}|I_e|^{1-rp}\Big(\sum_{i\in I_e} (s_{e,i}-s_{e,i-1})\varphi_e(f\ge s_{e,i-1})\Big)^{rp}d\eta(e)
\\ \ge~&\frac{2^{(1-r)p}}{\sup_e\|\varphi_e\|_\infty^{rp-p}|I_e|^{rp-1}}\int_E \big(\widehat{\varphi}_e(f)\big)^{rp}d\eta(e)
\end{align*}
where the second inequality is based on  the elementary inequality (see \cite[Lemma A.1]{Zhang25}) $$\big||b|^{t-1}b-|a|^{t-1}a\big|\ge|b-a|\left(\frac{|a|^t+|b|^t}{2}\right)^{1-\frac1t}$$
for $t\ge 1$ and $p\ge 1$ and $a,b\in\R$. Therefore, for any given $\widetilde{\llbracket \cdot\rrbracket}$ equipped with any degree  $\widetilde{\deg}:J\to\ (0,+\infty)$, 
$$ \frac{\Big(\int_E\big(\widehat{\varphi}_e(f^r)\big)^pd\eta(e)\Big)^{\frac1p}}{\Big(\widetilde{\llbracket f^r\rrbracket}_q
\Big)^{\frac1q} } \ge  \frac{2^{1-r}}{\sup\limits_e\|\varphi_e\|_\infty^{r-1}|I_e|^{r-\frac1p}}\frac{\Big(\int_E\big((\widehat{\varphi}_e(f)\big)^{pr}d\eta(e)\Big)^{\frac1p}}{\Big(\widetilde{\llbracket  f\rrbracket}_{qr}\Big)^{\frac1q}}$$
Similar to Step 1.2, we have
$$c(\Phi_p,\widetilde{\llbracket \cdot\rrbracket}_{q})\ge  \frac{2^{1-r}}{C_{\Phi}^{r-1}C_{JE}^{r-\frac1p}}\big(c(\Phi_{pr},\widetilde{\llbracket \cdot\rrbracket}_{qr})\big)^r$$ 
\end{proof}

In particular, if $|e_J|=2$ and $\|\varphi_e\|_\infty=1$ for any $e\in E$, i.e., $C_{\Phi}=C_{JE}=1$, we have $c(\Phi_p,\widetilde{\llbracket \cdot\rrbracket}_{q})\ge2^{1-r}\big(c(\Phi_{pr},\widetilde{\llbracket \cdot\rrbracket}_{qr})\big)^r$, yielding that for any $p,q\in [1,+\infty]$, the function 
$$r\mapsto \big(\frac12c(\Phi_{pr},\widetilde{\llbracket \cdot\rrbracket}_{qr})\big)^r$$
is decreasing with respect to $r\ge 1$.

\begin{prop}
Suppose that $\sup_{e\in E}|e_J|<\infty$ and $C_{\Phi}:=\sup_{e\in E}\|\varphi_e\|_\infty<\infty$. Then
$$c_k(\Phi_1,\widetilde{\llbracket \cdot\rrbracket}_{1})\ge  (2C_{\Phi})^{1-r}\big(c_k(\Phi_{r},\widetilde{\llbracket \cdot\rrbracket}_{r})\big)^r$$ 
\end{prop}
\begin{proof}
In this case, using a statement similar to Theorem \ref{thm:tilde-fg-equal}, it is not difficult to verify that the infimum related to $c(\Phi_1,\widetilde{\llbracket \cdot\rrbracket}_{1})$ can be taken on indicator functions, and for every indicator function, the corresponding $|I_e|$ is constant 1.
Thus, $C_{JE}=1$ and Proposition \ref{pro:phi-cpq21-r} concludes the desired inequality.  
\end{proof}

\noindent Case 2. using the disjoint-pair  Choquet integral under general setting  

Now we consider $\widetilde{\llbracket \cdot\rrbracket}_p$ instead of $\llbracket \cdot\rrbracket_p$. In this case, the infinity norm of $f$ with respect to $\psi_e$, i.e., $\|f\|_{\psi_e}=\inf\{t>0:\psi_e(f\ge s,f\le-s)= 0,\forall s\ge t\}$  is less than or equal to $\max_{x\in e_J}|f(x)|$. Therefore, $\|f^t\|_{\psi_e}^{q}=\|f\|_{\psi_e}^{qt}\le \max_{x\in e_J}|f(x)|^{qt}\le \sum_{x\in e_J}|f(x)|^{qt}$ and $$\int_E|f|_{\psi_e}^{qt}d\eta(e)\le \int_E\sum_{x\in e_J}|f(x)|^{qt}d\eta(e)\le\int_J\deg(x)|f(x)|^{qt}d\mu(x).$$  
This implies $c(\Psi_p,\widetilde{\llbracket \cdot\rrbracket}_{q})\le t\cdot  c(\Psi_{ps},\widetilde{\llbracket \cdot\rrbracket}_{qt})$. Below, we show an approach to enhance the degree term.

\begin{prop}
Suppose that $C_{JE}:=\sup_{e\in E}(|e_J|-1)<\infty$, and let $K_{p,q}=\big(\frac{C_{JE}^{p-1}}{2}\min \{2,C_{JE}\}\big)^{\frac{1}{pq'}}$. Then
$$c(\Psi_p,\widetilde{\llbracket \cdot\rrbracket}_{  (1-\frac1r)pq' })\le r   K_{p,q}\cdot  c(\Psi_{pq},\widetilde{\llbracket \cdot\rrbracket}_{(r-1)pq'})$$ 
\end{prop}
\begin{proof}The proof is very similar to that of Proposition \ref{pro:phi-rKcpq}. 
For any $e$, $\{\psi_e(f\ge s,f\le -s):s\ge0\}$ is a finite set, particularly, there exist $s_0:=0< s_1< \cdots< s_m< +\infty:=s_{m+1}$ such that $\psi_e(f\ge s,f\le -s)$ is constant on each interval $(s_{i-1},s_i)$, $i=1,\cdots,m+1$. Clearly, $\psi_e(f\ge s,f\le -s)=0$ when $s\ge s_m$. Let $I_e$ be the collection of the labels of   these $s_1,\cdots,s_m$, i.e., $I_e=\{1,\cdots,m\}$. 
Then,  
\begin{align*}
&\int_E\big(\widehat{\psi}_e(f^r)\big)^pd\eta(e)
\\=~& r^p\int_E\Big(\int_0^{\|f\|_{\psi_e}} s^{r-1}\psi_e(f\ge s,f\le -s)ds\Big)^pd\eta(e)  \\=~& \int_E\Big(\sum_{i\in I_e} \int_{s_{e,i-1}}^{s_{e,i}}rs^{r-1}\psi_e(f\ge s,f\le -s)ds\Big)^pd\eta(e)
\\=~&\int_E\Big(\sum_{i\in I_e} (s_{e,i}^r-s_{e,i-1}^r)\psi_e(f\ge s_{e,i-1},f\le -s_{e,i-1})\Big)^pd\eta(e)
\\\le~&\int_E\left(\sum_{i\in I_e}r(s_{e,i}-s_{e,i-1})\big(\frac{s_{e,i}^r+s_{e,i-1}^r}{2}\big)^{\frac{1}{r'}}\psi_e(f\ge s_{e,i-1},f\le -s_{e,i-1})\right)^pd\eta(e)
\\ \le~&
r^p\left(\int_E\Big(\sum_{i\in I_e}(s_{e,i}-s_{e,i-1})^q\psi_e^q(f\ge s_{e,i-1},f\le -s_{e,i-1})\Big)^pd\eta(e)\right)^{\frac1q}\times
\\&~~~~~~~~~~~~~~\times
\left(\int_E\Big(\sum_{i\in I_e} \big(\frac{s_{e,i}^r+s_{e,i-1}^r}{2}\big)^{\frac{q'}{r'}}\Big)^pd\eta(e)\right)^{\frac{1}{q'}}
\\ \le~&
r^p \left(\int_E\Big(\sum_{i\in I_e}(s_{e,i}-s_{e,i-1})\psi_e(f\ge s_{e,i-1},f\le -s_{e,i-1})\Big)^{qp}d\eta(e)\right)^{\frac1q}\times
\\&~~~~~~~~~~~~~~\times
\left(\int_E |I_e|^{p-1}\sum_{i\in I_e} \big(\frac{s_{e,i}^r+s_{e,i-1}^r}{2}\big)^{\frac{q'}{r'}p}d\eta(e)\right)^{\frac{1}{q'}}
\\ \le~&
r^p \left(\int_E\Big(\widehat{\psi}_e(f)\Big)^{qp}d\eta(e)\right)^{\frac1q}
\left(\int_E|I_e|^{p-1}\sum_{i\in I_e} \frac{s_{e,i}^{(r-1)q'p}+s_{e,i-1}^{(r-1)q'p}}{2}de\right)^{\frac{1}{q'}}
\\ \le ~&
r^p K_{p,q}\left(\int_E\Big(\widehat{\psi}_e(f)\Big)^{pq}d\eta(e)\right)^{\frac1q}
\left(\int_X  \deg(x)|f(x)|^{(r-1)q'p}dx\right)^{\frac{1}{q'}}
\end{align*}
The remainder of the proof is the same as that of Proposition \ref{pro:phi-rKcpq} and thus we omit it.
\end{proof}

\begin{prop}\label{pro:psi-cpq21-r}
Suppose that $K_{p,r}':=2^{\frac1r-1}(\sup\limits_{e\in E}\|\psi_e\|_\infty^{r-1}|I_e|^{r-\frac1p})^{-1}$. Then, for any degree and its related norm $\widetilde{\llbracket \cdot\rrbracket}$, we have
$$c(\Psi_p,\widetilde{\llbracket \cdot\rrbracket}_{q})\ge  K_{p,r}'\big(c(\Psi_{pr},\widetilde{\llbracket \cdot\rrbracket}_{qr})\big)^r.$$ 
\end{prop}

\begin{proof}
The proof is similar to that of Proposition \ref{pro:phi-cpq21-r}:
\begin{align*}
&\int_E\big(\widehat{\psi}_e(f^r)\big)^pd\eta(e)
\\ =~
&\int_E\Big(\sum_{i\in I_e} (s_{e,i}^r-s_{e,i-1}^r)\psi_e(f\ge s_{e,i-1},f\le -s_{e,i-1})\Big)^pd\eta(e)
\\ \ge~& \int_E \sum_{i\in I_e} (s_{e,i}^r-s_{e,i-1}^r)^p\psi_e(f\ge s_{e,i-1},f\le -s_{e,i-1}) ^pd\eta(e)
\\ \ge~&\int_E\sum_{i\in I_e} (s_{e,i}-s_{e,i-1})^p\left(\frac{s_{e,i}^r+s_{e,i-1}^r}{2}\right)^{(1-\frac1r)p}\psi_e(f\ge s_{e,i-1},f\le -s_{e,i-1})^pd\eta(e)
\\ \ge~&\int_E\frac{1}{ \|\psi_e\|_\infty^{rp-p}}\sum_{i\in I_e} (s_{e,i}-s_{e,i-1})^p\left(\frac{s_{e,i}-s_{e,i-1}}{2^{1/r}}\right)^{r(1-\frac1r)p}\psi_e(f\ge s_{e,i-1},f\le -s_{e,i-1})^{rp}d\eta(e)
\\ \ge~&\int_E\frac{2^{(\frac1r-1)p}}{ \|\psi_e\|_\infty^{rp-p}}\sum_{i\in I_e} (s_{e,i}-s_{e,i-1})^{rp}\psi_e(f\ge s_{e,i-1},f\le -s_{e,i-1})^{rp}d\eta(e)
\\ \ge~&\int_E \frac{2^{(\frac1r-1)p}}{ \|\psi_e\|_\infty^{rp-p}}|I_e|^{1-rp}\Big(\sum_{i\in I_e} (s_{e,i}-s_{e,i-1})\psi_e(f\ge s_{e,i-1},f\le -s_{e,i-1})\Big)^{rp}d\eta(e)
\\ \ge~&\frac{2^{(\frac1r-1)p}}{\sup_e\|\psi_e\|_\infty^{rp-p}|I_e|^{rp-1}}\int_E \big(\widehat{\psi}_e(f)\big)^{rp}d\eta(e)
\end{align*}
Therefore, in a manner similar to Proposition \ref{pro:psi-cpq21-r}, we have $$c(\Psi_p,\widetilde{\llbracket \cdot\rrbracket}_{q})\ge  \frac{2^{\frac1r-1}}{\sup\limits_e\|\psi_e\|_\infty^{r-1}|I_e|^{r-\frac1p}}\big(c(\Psi_{pr},\widetilde{\llbracket \cdot\rrbracket}_{qr})\big)^r$$
which completes the proof.
\end{proof}


{\footnotesize
\bibliographystyle{plain}
\bibliography{extension} 
}
 \end{document}